\documentclass[twoside, a4paper, 12pt]{article}
\usepackage{cite}
\usepackage{amsmath}
\usepackage{amssymb}
\usepackage{booktabs}
\usepackage{lastpage, graphicx, caption, subcaption, mathtools, bbm, dsfont}

\usepackage[hidelinks]{hyperref}
\usepackage{geometry}
\usepackage{tikz}
\usepackage{enumitem}
\usepackage{pgfplots}
\usepackage{pgfkeys}
\usepackage{setspace}
\usepackage{amsthm}
\usepackage{titlesec}
\usepackage{fancyhdr}

\titlelabel{\thetitle.\quad}
\theoremstyle{plain}
\newtheorem{lemma}{Lemma}[section]
\newtheorem{theorem}{Theorem}[section]

\newtheorem{proposition}{Proposition}[section]
\newtheorem*{corollary}{Corollary}

\theoremstyle{remark}
\newtheorem{remark}{\bf Remark}

\theoremstyle{definition}

\newtheorem{example}{Example}[section]
\newtheorem{Illustration}{Illustration}[section]

\makeatletter
\renewcommand{\maketitle}{
	\begin{center}
		\rule[1em]{\textwidth}{0mm}
		\baselineskip=0.30in
		{\bfseries \@title} \par
		\vspace{5mm}
		\baselineskip=0.2in
		{\bfseries \@author}\par
		\vspace{1mm}
		{\it \@address} \par
		{\small\tt \@email} \par
		\vspace{3mm}
		
	\end{center}
	\vspace{3mm}
}
\newcommand{\address}[1]{\def\@address{#1}}
\newcommand{\email}[1]{\def\@email{#1}}

\title{On the Clean Graph of a Ring}
\author{Randhir Singh$^{a,*}$, S. C. Patekar$^b$}
\address{$^a$Department of Mathematics, Savitribai Phule Pune University, Pune-411007, India\\
	$^b$Department of Mathematics, Savitribai Phule Pune University, Pune-411007, India\\
}
\email{$^a$randhir\textunderscore singh46@yahoo.com, $^b$shri82patekar@gmail.com}
\fancyhfoffset[L]{1cm} 
\fancyhfoffset[R]{1cm} 
\begin{document}
	\maketitle
	\begin{abstract}
		Let $R$ be a ring (not necessarily a commutative ring) with identity. The clean graph $Cl(R)$ of a ring $R$ is a graph with vertices in the form of an ordered pair $(e,u)$, where $e$ is an idempotent and $u$ is a unit of ring $R$, respectively. Two distinct vertices $(e,u)$ and $(f,v)$ are adjacent in $Cl(R)$ if and only if $ef=fe=0$ or $uv=vu=1$. In this study, we considered the induced subgraph $Cl_2(R)$ of $Cl(R)$. We determined the Wiener index of $Cl_2(R)$, and we showed $Cl_2(R)$ has a perfect matching. In addition, we determined the matching number of $Cl_2(R)$ if $|U(R)|$ is not even. 
		
	\end{abstract}
	
	\section{Introduction}
	Studying graphs related to algebraic structures is an interesting and long-standing way to gain insights into the different algebraic properties of graphs. This method allows the examination of various graph parameters through algebraic properties, leading to a more thorough understanding. For example, one can explore graphs associated with the ring $R$ with unity, which includes zero-divisor graphs and idempotent graphs.
	
	The zero-divisor graph of a commutative ring was first introduced by Beck \cite{IB}. Beck studied the problem of coloring the graph of a commutative ring with unity. The definition of the zero-divisor graph of a commutative ring with unity was later modified by Anderson et al. \cite{DP}. They defined the zero-divisor graph $\Gamma (R)$ of a commutative ring $R$ as a simple undirected graph. The vertex set $V(\Gamma(R))$ consists of non-zero zero-divisors of the commutative ring $R$ with unity, and two distinct vertices $x$ and $y$ are adjacent if and only if $xy=0$. Subsequently, many researchers have studied graphs associated with the algebraic structure (see \cite{Akb, DM, DLS}).
	
	In $1977$, Nicholson \cite{WK} introduced the concept of a clean ring, which has since attracted considerable research attention. An element of a ring is considered to be "clean" if it can be expressed as the sum of an idempotent and a unit. A ring is classified as clean if all its elements are clean. Akbari et al. \cite{Akb1} defined the Idempotent graph, denoted as $I(R)$, with nontrivial idempotents of ring $R$ as its vertex set. Two distinct vertices are adjacent in $I(R)$ if and only if their product is zero. In $2021$, Habibi et al. \cite{Habibi} introduced the concept of the clean graph to establish a connection between the properties of a clean ring and its graph-theoretic aspects. This development aims to enhance our understanding of the relationship between these concepts and their potential implications across various fields.
	
	The topological index of a graph is a numerical parameter that characterizes its topology by analyzing its connectivity patterns. These indices are alternatively referred to as molecular descriptors or connectivity indices. Topological indices play an important role in Chemistry and Computer science to study their physical properties using graph representation. In $1947$, Harold Wiener \cite{HW} introduced a topological index called the Wiener index. Wiener used the Wiener index to study the boiling point of Paraffin using a mathematical approach. Recently, many applications of the Wiener index arose; see \cite{Dinar, Eliasi, Nikoli, Spiro }.
	\section{Preliminary}
	In this study, we use several definitions and terminologies, which are as follows. 
	
	Throughout the study, $R$ is a commutative Artinian ring with unity. The ring of integers modulo $n$ is denoted by $\mathbb {Z}_n$. An element $e\in R$ is called \textit{idempotent} if $e^2=e$ and $u\in R$ is called \textit{unit} if there exists $v\in R$ such that $uv=vu=1$. The set of idempotents and unit elements of $R$ are denoted by $Id(R)$ and $U(R)$, respectively. Furthermore, $U'(R)=\{u\in U(R):u^2=1\}$ and $U''(R)=U(R)\backslash U'(R)$
	
	For graph $ G $, $V(G)$ denotes the set of vertices of graph $G$, and $E(G)$ denotes the set of edges of the graph $G$. If two vertices $u$ and $v$ are adjacent, we denote them by $u \sim v$.
	
	The distance between vertices $u$ and $v$ in $ G $, denoted by $d(u,v)$, is the number of edges on the shortest path between $u$ and $v$. If there is no path between $ u $ and $ v $, then $ d(u,v)=\infty $. The diameter of a graph $G$, denoted by $ \textit{diam(G)} $, is defined as the maximum distance between a pair of vertices. 
	
	The idempotent graph and the clean ring motivated Habibi et al.\cite{Habibi} to introduce the clean graph $Cl(R)$ with vertex set $V(Cl(R))=\{(e,u):e\in Id(R) \text{ and } u\in U(R)\}$. Two distinct vertices $(e,u),(f,v)\in V(Cl(R))$ are adjacent if and only if $ef=fe=0$ or $uv=vu=1$. The subgraph induced by the set $\{(e,u): e(\neq0)\in Id(R) \text{ and } u\in U(R)\}$ is denoted by $Cl_2(R)$. Throughout the study, we considered a clean graph as $Cl_2(R)$. 
	
	In this study, we determined the Wiener index of $Cl_2(R)$, the matching number, and perfect matching in $Cl_2(R)$. The Wiener index of the graph $G$ is a distance-based topological index introduced by H. Wiener\cite{HW} and defined to be the sum of all distances between all pairs of the vertices of graph $G$. Hosoya\cite{Hosoya} gave the mathematical representation of the Wiener index of graph $G$, defined by $W(G)=\frac{1}{2}\sum\limits_{v\in V(G)}d(u|G)$, where $d(u|G)=d(u,v)$ for any $v\in G$. A matching in a graph $G$ is a set $M$ of non-loop edges with no shared endpoints. A perfect matching in a graph is a matching that saturates every vertex. A matching number, denoted by $\mu(G)$, is the maximum size of a matching in $G$. A graph contains a perfect matching if $\mu(G)=\frac{|V|}{2}.$ Further, we determine the Wiener index of the clean graph $Cl_2(\mathbb{Z}_n)$ of ring of integers modulo $n$. The clean graph $Cl_2(R)$ of ring $R$ mainly deals with the units and idempotent. So, the number of units in $\mathbb{Z}_n$ is Euler's totient of $n$ and is denoted by $\phi(n)$. The Euler's phi (or totient) function of a positive integer $ n $ is the number of all positive integers less than $ n $ which are relatively prime to $ n $. If $n$ is a positive integer with prime factorization, $n = \prod_{i=1}^{k}{p_i}^{\alpha_i}$, then $\phi(n)=\prod_{i=1}^{k}({p_i}^{\alpha_i}-{p_i}^{\alpha_i-1})$. Hence, the set of units in $\mathbb{Z}_n$ is $U(\mathbb{Z}_n)=\{u_1=1,u_2,u_3,...u_r,u_{r+1},u_{r+2},...,u_{\phi(n)-1},u_{\phi(n)}\}$. 
	
	Let us recall some of the important results that play a crucial role in this study.
	\begin{proposition}[\textbf{\cite{Hewitt}}] \label{Prop 1}
		There are $2^k$ idempotents and $2^k-1$ non-zero idempotents in $\mathbb{Z}_n$, where $k$ is the number of distinct primes dividing $n$.

	\end{proposition}
	
	$Id({\mathbb{Z}_n})^*=\{ e_1,e_2,e_3,,...,e_{2^k-2},e_{2^k-1} \}$, where $e_1= 1$, $e_{2m}e_{2m+1}=0$ and $e_{2m+1} = 1 - e_{2m}$ for $m \in \{ 1,2, ...,\frac{2^{k}-4}{2} , \frac{2^{k}-2}{2} \}$
	
	\begin{theorem}[\textbf{\cite{Habibi}}]\label{T1}
		The $Cl_2(R)$ is connected if and only if $R$ has a nontrivial idempotent. Moreover, if $Cl_2(R)$ is connected, then $diam(Cl_2(R))\leq 3$.
	\end{theorem}
	In Section \ref{Sec 3}, we begin with a remark in which we partition the vertex set $V(Cl_2(R)$. In Section \ref{Sec4}, we determine the matching number of the clean graph $Cl_2(R)$ of the Artinian ring $R$. 
	
	\section{Wiener Index of $Cl_2(R)$}\label{Sec 3}
	Let $R$ be a commutative Artinian ring with unity. From \cite[Theorem 8.7]{Atiyah}, $R\cong R_1\times R_2\times\cdots\times R_n$, where $ R_i $ is a local ring for each $ i=1,2,...,n$. By \cite[Proposition 19.2]{Lam}, $Id(R_i)=\{0,1\}$ for each $i=1,2,...,n$. So that, all the idempotents are $ (0,1) $- vectors. Clearly, $ |Id(R)| = 2^n $ and $ |Id(R)^*|=|Id(R)\backslash\{0\}| = 2^n-1 $. In order to make an easy computation, we need the following remark. 
	
	\begin{remark}\label{remark 1}
		Since vertex set $V(Cl_2(R))$ of $Cl_2(R)$ is $V(Cl_2(R))=Id(R)^*\times U(R)$ and $|Id(R)^*|=2^n-1$. We arrange the idempotents and units in such way that $e_{2m}e_{2m+1}=0$ and $U(R)=\{u_i\in U'(R):u_i^2=1, i=1,2,\cdots,u_{|U'(R)|}\}\cup \left\{u_j,u_k\in U''(R): u_ju_k=1, j=|U'(R)|+1,|U'(R)|+2, \cdots, |U'(R)|+\frac{|U''(R)|}{2}\right\}$. In order to make easy computation, we partition the vertices of $Cl_2(R)$ as $V_1\sqcup V_2\sqcup...\sqcup V_{2^n-1}$ such that $|V_i|=|U'(R)|+|U''(R)|$ for each $i$, $V_{2m}=\{(e_{2m},u_j): e_{2m}\in Id(R)^*, u_j \in U(R)\}$ and $V_{2m+1}=\{(e_{2m+1},u_k): e_{2m+1}\in Id(R)^*, u_k \in U(R)\}$ for $m=1,2,\ldots,\frac{2^n-2}{2}$; see Figure \ref{Figure 1}.
		
		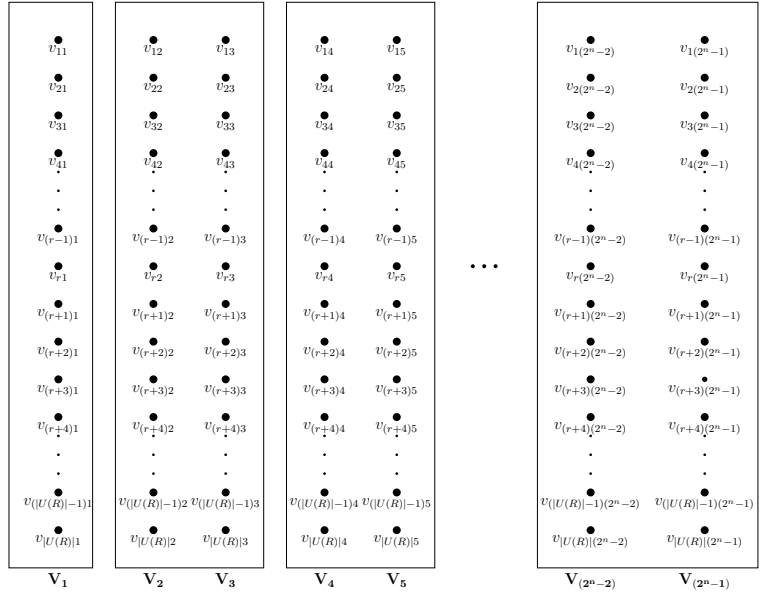
\begin{figure}[h!]
			
			\begin{center}
				\scalebox{0.5}{{\begin{tikzpicture} 
							\draw (-9.3,-7) rectangle (-7.1,8);
							\draw (-8,-7) node[below]{$\mathbf {V_1}$};
							\fill (-8,7) circle (3pt) node[below] {$v_{11}$}; 
							\fill (-8,6) circle (3pt) node[below] {$v_{21}$}; 
							\fill (-8,5) circle (3pt) node[below] {$v_{31}$}; 
							\fill (-8,4) circle (3pt) node[below] {$v_{41}$};
							\fill (-8,3.5) circle (1pt); 
							\fill (-8,3) circle (1pt); 
							\fill (-8,2.5) circle (1pt);
							\fill (-8,2) circle (3pt) node[below] {$v_{(r-1)1}$};
							\fill (-8,1) circle (3pt) node[below] {$v_{r1}$};
							\fill (-8,0) circle (3pt) node[below] {$v_{(r+1)1}$};
							\fill (-8,-1) circle (3pt) node[below] {$v_{(r+2)1}$};
							\fill (-8,-2) circle (3pt) node[below] {$v_{(r+3)1}$};
							\fill (-8,-3) circle (3pt) node[below] {$v_{(r+4)1}$};
							\fill (-8,-3.5) circle (1pt);
							\fill (-8,-4) circle (1pt);
							\fill (-8,-4.5) circle (1pt);
							\fill (-8,-5) circle (3pt) node[below] {$v_{(|U(R)|-1)1}$};
							\fill (-8,-6) circle (3pt) node[below] {$v_{|U(R)|1}$};
							

							\draw (-6.5,-7) rectangle (-2.6,8);
							\draw (-5.5,-7) node[below]{$\mathbf {V_2}$};
							\fill (-5.5,7) circle (3pt) node[below] {$v_{12}$}; 
							\fill (-5.5,6) circle (3pt) node[below] {$v_{22}$}; 
							\fill (-5.5,5) circle (3pt) node[below] {$v_{32}$}; 
							\fill (-5.5,4) circle (3pt) node[below] {$v_{42}$};
							\fill (-5.5,3.5) circle (1pt); 
							\fill (-5.5,3) circle (1pt); 
							\fill (-5.5,2.5) circle (1pt);
							\fill (-5.5,2) circle (3pt) node[below] {$v_{(r-1)2}$};
							\fill (-5.5,1) circle (3pt) node[below] {$v_{r2}$};
							\fill (-5.5,0) circle (3pt) node[below] {$v_{(r+1)2}$};
							\fill (-5.5,-1) circle (3pt) node[below] {$v_{(r+2)2}$};
							\fill (-5.5,-2) circle (3pt) node[below] {$v_{(r+3)2}$};
							\fill (-5.5,-3) circle (3pt) node[below] {$v_{(r+4)2}$};
							\fill (-5.5,-3.5) circle (1pt);
							\fill (-5.5,-4) circle (1pt);
							\fill (-5.5,-4.5) circle (1pt);
							\fill (-5.5,-5) circle (3pt) node[below] {$v_{(|U(R)|-1)2}$};
							\fill (-5.5,-6) circle (3pt) node[below] {$v_{|U(R)|2}$};

							\draw (-3.6,-7) node[below]{$\mathbf {V_3}$};
							\fill (-3.6,7) circle (3pt) node[below] {$v_{13}$}; 
							\fill (-3.6,6) circle (3pt) node[below] {$v_{23}$}; 
							\fill (-3.6,5) circle (3pt) node[below] {$v_{33}$}; 
							\fill (-3.6,4) circle (3pt) node[below] {$v_{43}$};
							\fill (-3.6,3.5) circle (1pt); 
							\fill (-3.6,3) circle (1pt); 
							\fill (-3.6,2.5) circle (1pt);
							\fill (-3.6,2) circle (3pt) node[below] {$v_{(r-1)3}$};
							\fill (-3.6,1) circle (3pt) node[below] {$v_{r3}$};
							\fill (-3.6,0) circle (3pt) node[below] {$v_{(r+1)3}$};
							\fill (-3.6,-1) circle (3pt) node[below] {$v_{(r+2)3}$};
							\fill (-3.6,-2) circle (3pt) node[below] {$v_{(r+3)3}$};
							\fill (-3.6,-3) circle (3pt) node[below] {$v_{(r+4)3}$};
							\fill (-3.6,-3.5) circle (1pt);
							\fill (-3.6,-4) circle (1pt);
							\fill (-3.6,-4.5) circle (1pt);
							\fill (-3.6,-5) circle (3pt) node[below] {$v_{(|U(R)|-1)3}$};
							\fill (-3.6,-6) circle (3pt) node[below] {$v_{|U(R)|3}$};
							\draw (-2,-7)  rectangle (1.9,8);
							\draw (-1,-7) node[below]{$\mathbf {V_4}$};
							\fill (-1,7) circle (3pt) node[below] {$v_{14}$}; 
							\fill (-1,6) circle (3pt) node[below] {$v_{24}$}; 
							\fill (-1,5) circle (3pt) node[below] {$v_{34}$}; 
							\fill (-1,4) circle (3pt) node[below] {$v_{44}$};
							\fill (-1,3.5) circle (1pt); 
							\fill (-1,3) circle (1pt); 
							\fill (-1,2.5) circle (1pt);
							\fill (-1,2) circle (3pt) node[below] {$v_{(r-1)4}$};
							\fill (-1,1) circle (3pt) node[below] {$v_{r4}$};
							\fill (-1,0) circle (3pt) node[below] {$v_{(r+1)4}$};
							\fill (-1,-1) circle (3pt) node[below] {$v_{(r+2)4}$};
							\fill (-1,-2) circle (3pt) node[below] {$v_{(r+3)4}$};
							\fill (-1,-3) circle (3pt) node[below] {$v_{(r+4)4}$};
							\fill (-1,-3.5) circle (1pt);
							\fill (-1,-4) circle (1pt);
							\fill (-1,-4.5) circle (1pt);
							\fill (-1,-5) circle (3pt) node[below] {$v_{(|U(R)|-1)4}$};
							\fill (-1,-6) circle (3pt) node[below] {$v_{|U(R)|4}$};

							\fill  (2.9,1) circle (1.5pt)(3.2,1) circle (1.5pt)(3.5,1)  circle (1.5pt);

							
							\draw(4.6,-7) rectangle (10.4,8);
							\draw (0.9,-7) node[below]{$\mathbf {V_{5}}$};
							\fill (0.9,7) circle (3pt) node[below] {$v_{15}$}; 
							\fill (0.9,6) circle (3pt) node[below] {$v_{25}$};
							\fill (0.9,5) circle (3pt) node[below] {$v_{35}$}; 
							\fill (0.9,4) circle (3pt) node[below] {$v_{45}$};
							\fill (0.9,3.5) circle (1pt); 
							\fill (0.9,3) circle (1pt); 
							\fill (0.9,2.5) circle (1pt);
							\fill (0.9,2) circle (3pt) node[below] {$v_{(r-1)5}$};
							\fill (0.9,1) circle (3pt) node[below] {$v_{r5}$};
							\fill (0.9,0) circle (3pt) node[below] {$v_{(r+1)5}$};
							\fill (0.9,-1) circle (3pt) node[below] {$v_{(r+2)5}$};
							\fill (0.9,-2) circle (3pt) node[below] {$v_{(r+3)5}$};
							\fill (0.9,-3) circle (3pt) node[below] {$v_{(r+4)5}$};
							\fill (0.9,-3.5) circle (1pt);
							\fill (0.9,-4) circle (1pt);
							\fill (0.9,-4.5) circle (1pt);
							\fill (0.9,-5) circle (3pt) node[below] {$v_{(|U(R)|-1)5}$};
							\fill (0.9,-6) circle (3pt) node[below] {$v_{|U(R)|5}$};
							
							
							\draw (6,-7) node[below]{$\mathbf {V_{(2^n-2)}}$};
							\fill (6,7) circle (3pt) node[below] {$v_{1{(2^n-2)}}$}; 
							\fill (6,6) circle (3pt) node[below] {$v_{2{(2^n-2)}}$};
							\fill (6,5) circle (3pt) node[below] {$v_{3(2^n-2)}$}; 
							\fill (6,4) circle (3pt) node[below] {$v_{4(2^n-2)}$};
							\fill (6,3.5) circle (1pt); 
							\fill (6,3) circle (1pt); 
							\fill (6,2.5) circle (1pt);
							\fill (6,2) circle (3pt) node[below] {$v_{(r-1)(2^n-2)}$};
							\fill (6,1) circle (3pt) node[below] {$v_{r(2^n-2)}$};
							\fill (6,0) circle (3pt) node[below] {$v_{(r+1)(2^n-2)}$};
							\fill (6,-1) circle (3pt) node[below] {$v_{(r+2)(2^n-2)}$};
							\fill (6,-2) circle (3pt) node[below] {$v_{(r+3)(2^n-2)}$};
							\fill (6,-3) circle (3pt) node[below] {$v_{(r+4)(2^n-2)}$};
							\fill (6,-3.5) circle (1pt);
							\fill (6,-4) circle (1pt);
							\fill (6,-4.5) circle (1pt);
							\fill (6,-5) circle (3pt) node[below] {$v_{(|U(R)|-1)(2^n-2)}$};
							\fill (6,-6) circle (3pt) node[below] {$v_{|U(R)|(2^n-2)}$};
							
							\draw (9,-7) node[below]{$\mathbf {V_{(2^n-1)}}$};
							\fill (9,7) circle (3pt) node[below] {$v_{1{(2^n-1)}}$}; 
							\fill (9,6) circle (3pt) node[below] {$v_{2{(2^n-1)}}$};
							\fill (9,5) circle (3pt) node[below] {$v_{3(2^n-1)}$}; 
							\fill (9,4) circle (3pt) node[below] {$v_{4(2^n-1)}$};
							\fill (9,3.5) circle (1pt); 
							\fill (9,3) circle (1pt); 
							\fill (9,2.5) circle (1pt);
							\fill (9,2) circle (3pt) node[below] {$v_{(r-1)(2^n-1)}$};
							\fill (9,1) circle (3pt) node[below] {$v_{r(2^n-1)}$};
							\fill (9,0) circle (3pt) node[below] {$v_{(r+1)(2^n-1)}$};
							\fill (9,-1) circle (3pt) node[below] {$v_{(r+2)(2^n-1)}$};
							\fill (9,-2) circle (2pt) node[below] {$v_{(r+3)(2^n-1)}$};
							\fill (9,-3) circle (3pt) node[below] {$v_{(r+4)(2^n-1)}$};
							
							\fill (9,-3.5) circle (1pt);
							\fill (9,-4) circle (1pt);
							\fill (9,-4.5) circle (1pt);
							\fill (9,-5) circle (3pt) node[below] {$v_{(|U(R)|-1)(2^n-1)}$};
							\fill (9,-6) circle (3pt) node[below] {$v_{|U(R)|(2^n-1)}$};
							
						\end{tikzpicture}\label{figure 4}}}
				
				\caption{Partition of $V(\mathbf{R})$}\label{Figure 1}\end{center}
		\end{figure}
		
	\end{remark} 
	
	To determine the Wiener index $W(Cl_2(R)$ of the clean graph $Cl_2(R)$ of a ring, we require the following lemma:
	\begin{lemma}\label{Lemma 0.1}
		Let the clean graph $Cl_2(R)$ of a commutative ring $R$ with unity be connected graph. Then for $(e_i,u_j), (e_k,u_l) \in V(Cl_2(R))$
		\begin{center}
			$d((e_i,u_j), (e_k,u_l))=\begin{cases}
				3 & \text{ if } e_i=e_k=1 \text{ and } u_ju_l\neq 1\\
				2 & \text{ if } e_i=e_k\neq 1 \text{ and } u_ju_l\neq 1\\
				2 & \text{ if } e_i\neq e_k, u_ju_l\neq 1 \text{ and } e_ie_k\neq0\\
				1 & \text{ if } e_i\neq e_k, e_ie_k=0 \text{ or } u_iu_l=1\\
				1 & \text{ if } e_i=e_k=1, u_iu_l=1.
			\end{cases}$
		\end{center}

	\end{lemma}
	\begin{proof}
		Let $ R = \prod_{i=1}^n R_i$, where each $R_i $ is a local ring for $ i = 1, 2, \ldots, n $. Given that $ |U(R)| = 2 $, it follows that $ U(R) = U'(R) $. According to \cite[Theorem 2.2]{Habibi}, if $n\geq 2$, $ Cl_2(R) $ is connected with $ Diam(Cl_2(R)) \leq 3 $. Furthermore, if $ n \geq 2 $ and $ |U(R)| = 2 $, then $ Diam(Cl_2(R)) = 3 $, as  $(1,1) \sim (e_i,1) \sim (1-e_i,u_2) \sim (1,u_2)$.
		
		Now, we find the distance between the vertices $(e_i,u_j), (e_k,u_k)\in V(Cl_2(R))$ for the following cases:
		
		\noindent \textbf{Case 1.} If $e_i=e_k=1$ and $u_ju_l\neq1$ then there exist $(e_m,u_j^{-1}), (1-e_m,u_k^{-1})\in V(Cl_2(R))$ such that $(1,u_j)\sim(e_m,u_j^{-1})\sim(1-e_m,u_k^{-1})\sim(1,u_k)$ as $u_ju_j^{-1}=u_ku_k^{-1}=1$ and $e_m(1-e_m)=0$. Therefore, $d((1,u_j),(1,u_l))=3$.
		
		\noindent\textbf{Case 2.} If $e_i=e_k\neq 1$ and $u_ju_l\neq1$ then there exists $(1-e_i,u_p)\in V(Cl_2(R))$. Such that $(e_i, u_j)\sim (1-e_i, u_p)\sim(e_i,u_l)$ as $e_i(1-e_i)=0$. Therefore, $d((e_i,u_j),(e_i,u_l))=2$.
		
		\noindent\textbf{Case 3.} If $e_i\neq e_k$ and $u_ju_l\neq 1$ then there exists $(1-e_k,u_j^{-1})\in V(Cl_2(R))$ such that $(e_i, u_j)\sim (1-e_k,u_j^{-1})\sim (e_k, u_l)$. This leads to $d((e_i, u_j),(e_k,u_l))=2$ for $e_i\neq e_k\neq 1$ and $u_ju_l\neq 1$.
		
		\noindent \textbf{Case 4} and \textbf{Case 5 } are trivial as $e_ie_k=0$ and $u_ju_l=1$ that imply $ d((e_i,u_j),(e_k,u_l)) =1$.
	\end{proof}
	
	\begin{theorem}
		Let $R$ be a commutative ring with unity. Then 
		\begin{enumerate}
			\item $\displaystyle W(Cl_2(R))=\frac{1}{2}(2^n-2)(2^n-1), \text{ if }|U''(R)|=0 \text{ and }|U'(R)|=1 $
			\item $\displaystyle W(Cl_2(R))=\frac{1}{2}\left[|U'(R)|^2(2\cdot2^{2n}-17\cdot2^n+21)-|U'(R)|(2\cdot2^{2n}-8\cdot2^n+11)\right],\text{ if } |U''(R)|=0 \text{ and }|U'(R)|>1$
			\item $\frac{1}{2}\left[|U(R)|^2(2\cdot2^{2n}-5\cdot2^n+5)-|U(R)|(2^{2n}-2^n+3)+|U'(R)|2^n\right], \text{ if } |U''(R)|\neq0 \text{ and }\\|U'(R)|\geq 1$.
		\end{enumerate}
		
	\end{theorem}
	\begin{proof} Let $R=R_1\times R_2 \times\cdot\cdot\cdot\times R_n$, where each $R_i$ is a local ring.  Let $Id(R)^*=\{e_i\in R:e_i\neq0,e_i(e_i-1)=0,e_1=1 \}$ and $U(R)=\{u_i\in R: \text{ there exists a unique }u_j \in R\text{ such that }u_iu_j=1 \}$. To find the Wiener index $W(Cl_2(R))$ of $Cl_2(R)$, we make the following cases:
		
		\noindent\textbf{Case 1.} When $|U''(R)| = 0$ and $|U'(R)| = 1$: Here, $|U'(R)| = 1$ implies $U(R) = U'(R) = \{1\}$. Consequently, $V(Cl_2(R)) = \{(e_i,1) : e_i \in Id(R)^*$ and $1 \in U'(R)\}$. Since $1^2=1$, the vertices $(e_i,1)\sim(e_j,1)$ for all $e_i, e_j\in Id(R)^*$. Therefore, $Cl_2(R)\cong K_{Id(R)-1}.$ Thus, $W(Cl_2(R))=\frac{1}{2}(2^n-2)(2^n-1).$ 
		
		\noindent\textbf{Case 2.} When $|U''(R)| = 0$ and $|U'(R)| > 1$: Here, $|U''(R)| = 0$ and $|U'(R)| > 1$ imply $U(R) = U'(R) = \{u_i \in U(R) : u_i^2 = 1\}$. The vertex set of $Cl_2(R)$ is $V(Cl_2(R))=\{(e_i, u_j):e_i\in Id(R)^*, u_j^2=1\}.$ As $e_i\in Id(R)^*$ and $e_i-1\in Id(R)^*$ are orthogonal, we partition $V(Cl_2(R))$ as $V_1\sqcup V_2\sqcup...\sqcup V_{2^n-1}$ such that $|V_i|=|U'(R)|+|U''(R)|$ for each $i$, $V_{2m}=\{(e_{2m},u_j): e_{2m}\in Id(R)^*, u_j \in U(R)\}$ and $V_{2m+1}=\{(e_{2m+1},u_k): e_{2m+1}\in Id(R)^*, u_k \in U(R)\}$ for $m=1,2,\ldots,\frac{2^n-2}{2}$; see Figure \ref{Figure 1}. Now, we compute the Wiener index $W(Cl_2(\mathbb{Z}_n))$ of $Cl_2(R)$.

		\begin{align*}
			\displaystyle W(Cl_2(R))&=\sum\limits_{\substack{e_i,e_k\in Id(R)^*\\u_j,u_l\in U'(R)}}d((e_i,u_j),(e_k,u_l))\\
			&=\sum\limits_{\substack{l=2\\j<l}}^{|U'(R)|}\sum\limits_{j=1}^{|U'(R)|-1}d((e_1,u_j),(e_1,u_l))+\sum\limits_{i=2}^{2^n-1}\sum\limits_{\substack{l=2\\j<l}}^{|U'(R)|}\sum\limits_{j=1}^{|U'(R)|-1}d((e_i,u_j),(e_i,u_l))\\
			&\quad+\sum\limits_{j=2}^{2^n-1}\sum\limits_{l=1}^{|U'(R)|}\sum\limits_{j=1}^{U'(R)}d((e_1,u_j),(e_k,u_l))+\sum_{\substack{e_ie_k=0}}\sum\limits_{l=1}^{|U'(R)|}\sum\limits_{j=1}^{|U'(R)|}d((e_i,u_j),(e_k,u_l))\\
			&\quad+\sum_{\substack{e_ie_k\neq0}}\sum\limits_{l=1}^{|U'(R)|}\sum\limits_{j=1}^{|U'(R)|}d((e_i,u_j),(e_k,u_l)).
		\end{align*}
		Let us define the following terms, 
		\begin{align*}
			&T_1=\sum\limits_{\substack{l=2\\j<l}}^{|U'(R)|}\sum\limits_{j=1}^{|U'(R)|-1}d((e_1,u_j),(e_1,u_l)),\\ &T_2=\sum\limits_{i=2}^{2^n-1}\sum\limits_{\substack{l=2\\j<l}}^{|U'(R)|}\sum\limits_{j=1}^{|U'(R)|-1}d((e_i,u_j),(e_i,u_l)),\\  &T_3=\sum\limits_{j=2}^{2^n-1}\sum\limits_{l=1}^{|U'(R)|}\sum\limits_{j=1}^{U'(R)}d((e_1,u_j),(e_k,u_l)),\\ &T_4=\sum_{\substack{e_ie_k=0}}\sum\limits_{l=1}^{|U'(R)|}\sum\limits_{j=1}^{|U'(R)|}d((e_i,u_j),(e_k,u_l)),\\  
			&\displaystyle  T_5=\sum_{\substack{e_ie_k\neq0}}\sum\limits_{l=1}^{|U'(R)|}\sum\limits_{j=1}^{|U'(R)|}d((e_i,u_j),(e_k,u_l)).
		\end{align*}
		Now, let us evaluate each summation individually
		\begin{enumerate}
			\item $ \displaystyle T_1=\sum\limits_{\substack{l=2\\j<l}}^{|U'(R)|}\sum\limits_{j=1}^{|U'(R)|-1}d((e_1,u_j),(e_1,u_l)) $
			\begin{align*}
				T_1=\sum\limits_{\substack{l=2\\j<l}}^{|U'(R)|}\{d((e_i,u_1),(e_i,u_l))+d((e_i,u_2),(e_i,u_l))+\ldots+d((e_i,u_{|U(R)|-1}),(e_i,u_l ))\}
			\end{align*}
			Using Lemma \ref{Lemma 0.1}, we have $ d((e_1,u_j),(e_1,u_l))=3$ if $e_i=e_k=1$ and $u_ju_l\neq1$, therefore
			\begin{align*}
				T_1&=3((|U'(R)|-1)+(|U'(R)|-2)+\ldots+1)\\
				&=3\left(\frac{|U'(R)|(|U'(R)|-1)}{2}\right).
			\end{align*}	
			\item $T_2=\sum\limits_{i=2}^{2^n-1}\sum\limits_{\substack{k=1\\j<k}}^{|U'(R)|-1}\sum\limits_{j=2}^{|U'(R)|}d((e_i,u_j),(e_i,u_k)):$
			\begin{align*}
				T_2=\sum\limits_{i=2}^{2^n-1}\sum\limits_{\substack{l=2\\j<l}}^{|U'(R)|}\{d((e_i,u_1),(e_i,u_l))+d((e_i,u_2),(e_i,u_l))+\ldots+d((e_i,u_{|U(R)|-1}),(e_i,u_l ))\}
			\end{align*}
			From Lemma \ref{Lemma 0.1}, we have $ d((e_i,u_j),(e_i,u_l))=3$ if $e_i=e_k\neq1$ and $u_ju_l\neq1$, therefore
			\begin{align*}
				T_2&=\sum\limits_{i=2}^{2^n-1}(2(|U'(R)|-1)+2(|U'(R)|-2)+\ldots+2)\\
				&=2(2^n-2)\frac{|U'(R)|(|U'(R)|-1)}{2}\\
				&=(2^n-2)|U'(R)|(|U'(R)|-1).
			\end{align*}
			\item $ T_3=\sum\limits_{j=2}^{2^n-1}\sum\limits_{k=1}^{|U'(R)|}\sum\limits_{i=1}^{|U'(R)|}d((e_1,u_j),(e_k,u_l)) :$
			
			From Lemma \ref{Lemma 0.1}, we have $ d((e_i,u_j),(e_k,u_l)) =\begin{cases}
				2& \text{ if } e_i\neq e_k, e_ie_k\neq 0 \text{ and } u_ju_l\neq1\\
				1 & \text{ if }  e_ie_k= 0 \text{ or } u_ju_l=1.
			\end{cases}$\\
			Since $u_j,u_l\in U'(R)$, it follows that $u_ju_l=1$ if and only if $u_j=u_l$. 
			Therefore, we split $T_3$ into two parts as follows.
			
			
			\begin{align*}
				T_3&=\sum\limits_{k=2}^{2^n-1}\sum\limits_{j=1}^{|U'(R)|}d((e_1, u_j),(e_k,u_j))+\sum\limits_{k=2}^{2^n-1}\sum\limits_{\substack{u_l\in U'(R)\\l\neq j}}\sum\limits_{u_j\in U'(R)}d((e_1, u_j),(e_k,u_l))\\
				&=\sum\limits_{k=2}^{2^n-1}|U'(R)|+\sum\limits_{k =2}^{2^n-1}2|U'(R)|(|U'(R)|-1)\\
				&=(2^n-2)|U'(R)|+(2^n-2)2|U'(R)|(|U'(R)|-1)\\
				&=(2^n-2)(2|U'(R)|^2-|U'(R)|).
			\end{align*}

			\item $ T_4=\sum\limits_{\substack{e_ie_k=0}}\sum\limits_{k=1}^{|U'(R)|}\sum\limits_{i=1}^{|U'(R)|}d((e_i,u_j),(e_k,u_l)) :$\label{T_4}
			
			From Lemma \ref{Lemma 0.1}, $d((e_i,u_j),(e_k,u_l))=1$ for all $u_j,u_l\in U'(R)$ and $e_ie_k=0$. As we have mentioned above, if $k=2m+1$ and $i=2m$, then $e_{2m}e_{2m+1}=0$ for $m=1,2,\ldots,\frac{2^n-2}{2}$.
			Note that, there are $\left(\frac{2^n-2}{2}\right)$ pair of idempotents $e_{2m},e_{2m+1}\in Id(R)^*$ such that $e_{2m}e_{2m+1}=0$. Therefore,
			
			\begin{align}
				\begin{split}
					T_4&=\sum\limits_{\substack{e_{2m},e_{2m+1}\in Id(R)^*}}\sum\limits_{\substack{u_l\in U(R)\\l=1}}^{|U(R)|}\sum\limits_{\substack{u_j\in U(R)\\j=1}}^{|U(R)|}d((e_i,u_j),(e_k,u_l))\\
					&=\left(\frac{2^n-2}{2}\right)|U'(R)|^2.
				\end{split}
			\end{align}
			
			\item $\displaystyle  T_5=\sum\limits_{\substack{e_ie_k\neq0\\e_i\neq e_k\neq 1}}\sum\limits_{l=1}^{|U'(R)|}\sum\limits_{j=1}^{|U'(R)|}d((e_i,u_j),(e_k,u_l)): $
			
			From Lemma \ref{Lemma 0.1}, we have $d((e_i,u_j),(e_k,u_l))=\begin{cases}
				2 & \text{ if } e_i\neq e_k, u_ju_l\neq 1 \text{ and } e_ie_k\neq0\\
				1 & \text{ if } e_i\neq e_k, e_ie_k=0 \text{ or } u_ju_l=1
			\end{cases}$
			
			To compute $T_5$, we divide the computation into two parts: $u_j=u_l$ and $u_j\neq u_l$ as $u_j,u_l\in U'(R)$. Using the same reasoning as in the previous cases, we have $d((e_i,u_j),(e_k,u_l))=d((e_i-1,u_j),(e_k,u_l))$ for $e_i\neq e_k\neq e_i-1$. Therefore,
			
			\begin{align*}
				T_5&=\sum\limits_{\substack{e_ie_k\neq0\\e_i\neq e_k\neq 1}}\sum\limits_{j=1}^{|U'(R)|}d((e_i,u_j),(e_k,u_j))+\sum\limits_{\substack{u_l\\u_l\neq u_j}}\sum\limits_{\substack{u_j\\u_j\neq u_l}}\sum\limits_{\substack{e_ie_k\neq0\\e_i\neq e_k\neq 1}}d((e_i,u_j),(e_k,u_l))\\
				&\quad+\sum\limits_{\substack{e_ie_k\neq0\\e_i\neq e_k\neq 1}}\sum\limits_{j=1}^{|U'(R)|}d((e_i-1,u_j),(e_k,u_j))+\sum\limits_{\substack{u_l\\u_l\neq u_j}}\sum\limits_{\substack{u_j\\u_j\neq u_l}}\sum\limits_{\substack{e_ie_k\neq0\\e_i\neq e_k\neq 1}}d((e_i-1,u_j),(e_k,u_l)).
			\end{align*} 
			Since, $d((e_i,u_j),(e_k,u_l))=d((e_i-1,u_j),(e_k,u_l)).$ Therefore,
			
			\begin{align*}
				T_5=2\sum\limits_{\substack{e_ie_k\neq0\\e_i\neq e_k\neq 1}}\sum\limits_{j=1}^{|U'(R)|}d((e_i,u_j),(e_k,u_j))+2\sum\limits_{\substack{u_l\\u_l\neq u_j}}\sum\limits_{\substack{u_j\\u_j\neq u_l}}\sum\limits_{\substack{e_ie_k\neq0\\e_i\neq e_k\neq 1}}d((e_i,u_j),(e_k,u_l)).
			\end{align*}
			
			Say, $S_1=2\sum\limits_{\substack{e_ie_k\neq0\\e_i\neq e_k\neq 1}}\sum\limits_{j=1}^{|U'(R)|}d((e_i,u_j),(e_k,u_j))$ and $S_2=2\sum\limits_{\substack{u_l\\u_l\neq u_j}}\sum\limits_{\substack{u_j\\u_j\neq u_l}}\sum\limits_{\substack{e_ie_k\neq0\\e_i\neq e_k\neq 1}}d((e_i,u_j),(e_k,u_l)).$  
			
			\begin{align*}
				S_1&= 2\sum\limits_{\substack{e_ie_k\neq0\\e_i\neq e_k\neq 1}}\sum\limits_{j=1}^{|U'(R)|}d((e_i,u_j),(e_k,u_j))\\
				&=2\sum\limits_{\substack{e_ie_k\neq0\\e_i\neq e_k\neq 1\\i<k\leq2^n-1}}\sum\limits_{\substack{i=2m+1\\m=1}}^{\frac{2^n-4}{2}}\sum\limits_{j=1}^{|U'(R)|}d((e_i,u_j),(e_k,u_j))\\
				&=2[(2^n-4)+(2^n-6)+\dots +2]|U'(R)|\\
				&=4[(2^{n-1}-2)+(2^{n-1}-3)+\dots +1]|U'(R)|\\
				&=4\left[\frac{(2^{n-1}-2)(2^{n-1}-1)}{2}\right]\left|U'(R)\right|\\
				&= \frac{(2^n-4)(2^n-2)}{2}\left|U'(R)\right|.
			\end{align*}
			and
			\begin{align*}
				S_2&=2\sum\limits_{\substack{u_l\\u_l\neq u_j}}\sum\limits_{\substack{u_j\\u_j\neq u_l}}\sum\limits_{\substack{e_ie_k\neq0\\e_i\neq e_k\neq 1}}d((e_i,u_j),(e_k,u_l))\\
				&= 2(|U'(R)|-1)|U'(R)|2\left[(2^n-4)+(2^n-6)+\dots +2\right]\\
				&= (2^n-4)(2^n-2)(|U'(R)|-1)|U'(R)|.
			\end{align*}
			
			Combining  $ S_1 $  and $ S_2 $  yields
			\begin{align}
				T_5=\frac{(2^n-4)(2^n-2)}{2}(2|U'(R)|^2-|U'(R)|).
			\end{align}

		\end{enumerate}
		
		Let us recall the simplified expressions for each term
		\begin{align*}
			T_1 &= 3 \left(\frac{|U'(R)|(|U'(R)| - 1)}{2}\right), \\
			T_2 &= (2^n - 2)|U'(R)|(|U'(R)| - 1), \\
			T_3 &= (2^n - 2)(2|U'(R)|^2 - |U'(R)|), \\
			T_4 &= \left(\frac{2^n - 2}{2}\right)|U'(R)|^2, \\
			T_5 &= \frac{(2^n-4)(2^n-2)}{2}(2|U'(R)|^2-|U'(R)|).
		\end{align*}
		
		Combining $ T_1 $, $ T_2 $, $ T_3 $, $ T_4 $, and $ T_5 $ yields
		
		\begin{align*}
			W(Cl_2(R)) &= 3 \left(\frac{|U'(R)|(|U'(R)| - 1)}{2}\right) + (2^n - 2)|U'(R)|(|U'(R)| - 1) \\
			&\quad + (2^n - 2)(2|U'(R)|^2 - |U'(R)|) + \left(\frac{2^n - 2}{2}\right)|U'(R)|^2 \\
			&\quad + \frac{(2^n-4)(2^n-2)}{2}(2|U'(R)|^2-|U'(R)|).
		\end{align*}
		
		Simplifying further, we get
		
		\begin{center}
			$ W(Cl_2(R))=\frac{1}{2}\left[|U'(R)|^2(2\cdot2^{2n}-5\cdot2^n+5)-|U'(R)|(2^{2n}-2\cdot2^n+3)\right] $
		\end{center}
		
		\noindent\textbf{Case 3.} When $|U''(R)|\neq0 \text{ and }|U'(R)|\geq1$ imply $|U''(R)|\geq2$: We have $Id(R)^*=\{e_i\in R:e_i\neq0, e_i^2=e_i \text{ and }e_{2i}e_{2i+1}=0 \text{ for } i=1,2,...,\frac{2^n-2}{2} \}$. On similar line as in previous case, we partition $V(Cl_2(R))$ as $V_1\sqcup V_2\sqcup...\sqcup V_{2^n-1}$ such that $|V_i|=|U'(R)|+|U''(R)|$ for each $i$, $V_{2m}=\{(e_{2m},u_j): e_{2m}\in Id(R)^*, u_j \in U(R)\}$ and $V_{2m+1}=\{(e_{2m+1},u_k): e_{2m+1}\in Id(R)^*, u_k \in U(R)\}$ for $m=1,2,\ldots,\frac{2^n-2}{2}$. Now, we compute the Wiener index $W(Cl_2(\mathbb{Z}_n))$ of $Cl_2(R)$ as follows:
		
		\begin{align*}
			W(Cl_2(R))&=\sum\limits_{\substack{(e_i,u_j),(e_k,u_k)\in V(Cl_2(R))}}d((e_i,u_j),(e_k,u_l))\\
			&=\sum\limits_{\substack{e_i\in Id(R)^*\\i=1}}^{2^n-1}\sum\limits_{\substack{u_j\in U(R)}}\sum\limits_{\substack{u_l\in U(R)}}d((e_i,u_j),(e_i,u_l))+\sum\limits_{\substack{e_i,e_k\in Id(R)^*\\e_i\neq e_k}}\sum\limits_{\substack{u_j\in U(R)}}\sum\limits_{\substack{u_l\in U(R)}}d((e_i,u_j),(e_k,u_l))\\
			&= \sum\limits_{\substack{u_l\in U(R)\\j<l\\l=2}}^{|U(R)|} \sum\limits_{\substack{u_j\in U(R)\\j=1}}^{|U(R)|-1}d((e_1,u_j),(e_1,u_l))+\sum\limits_{\substack{e_i\in Id(R)^*\\i=2}}^{2^n-1}\sum\limits_{\substack{u_l\in U(R)\\j<l\\l=2}}^{|U(R)|} \sum\limits_{\substack{u_j\in U(R)\\j=1}}^{|U(R)|-1}d((e_i,u_j),(e_i,u_l))\\
			&\quad+\sum\limits_{\substack{e_1,e_k\in Id(R)^*\\e_1\neq e_k}}\sum\limits_{\substack{u_j\in U(R)}}\sum\limits_{\substack{u_l\in U(R)}}d((e_1,u_j),(e_k,u_l))+\sum\limits_{\substack{e_i,e_k\in Id(R)^*\\e_i\neq e_k\neq e_1}}\sum\limits_{\substack{u_j\in U(R)}}\sum\limits_{\substack{u_l\in U(R)}}d((e_i,u_j),(e_k,u_l)).
		\end{align*}
		
		For convenience, let us say 
		
		\begin{align*}
			&S_1=\sum\limits_{\substack{u_l\in U(R)\\j<l\\l=2}}^{|U(R)|} \sum\limits_{\substack{u_j\in U(R)\\j=1}}^{|U(R)|-1}d((e_1,u_j),(e_1,u_l)),~ S_2=\sum\limits_{\substack{e_i\in Id(R)^*\\i=2}}^{2^n-1}\sum\limits_{\substack{u_l\in U(R)\\j<l\\l=2}}^{|U(R)|} \sum\limits_{\substack{u_j\in U(R)\\j=1}}^{|U(R)|-1}d((e_i,u_j),(e_i,u_l)),\\ & S_3=\sum\limits_{\substack{e_1,e_k\in Id(R)^*\\e_1\neq e_k}}\sum\limits_{\substack{u_l\in U(R)}}\sum\limits_{\substack{u_j\in U(R)}}d((e_1,u_j),(e_k,u_l)),~ S_4=\sum\limits_{\substack{e_i,e_k\in Id(R)^*\\e_i\neq e_k\neq e_1}}\sum\limits_{\substack{u_l\in U(R)}}\sum\limits_{\substack{u_j\in U(R)}}d((e_i,u_j),(e_k,u_l)).
		\end{align*}\\
		
		Now, let us compute each summation individually
		\begin{enumerate}
			\item $\displaystyle S_1= \sum\limits_{\substack{u_l\in U(R)\\j<l\\l=2}}^{|U(R)|} \sum\limits_{\substack{u_j\in U(R)\\j=1}}^{|U(R)|-1}d((e_1,u_j),(e_1,u_l))$

			As $U(R)=U'(R)\sqcup U''(R)$, $\displaystyle V_i=\{(e_i,u_j):e_i\in Id(R)^*, u_j\in U'(R)\}\sqcup\{(e_i,u_k):e_i\in Id(R)^*, u_k\in U''(R)\}$. By Lemma \ref{Lemma 0.1}, $d((e_i,u_j),(e_i,u_k))=\begin{cases}
				3 & \text{ if } e_i=1\text{ and }  u_ju_k\neq1\\
				1 & \text{ if } e_i=1\text{ and } u_ju_k=1.
			\end{cases}$
			
			To determine $ S_1 $, we first observe that $ d((e_1, u_j), (e_1, u_k)) = 1 $ for each pair of vertices $(e_1, u_j)$ and $(e_1, u_j^{-1})$ in $ V_1 $. The distance $ d((e_1, u_j), (e_1, u_k)) = 3 $ for each pair $(e_1, u_j)$ and $(e_1, u_k)$ in $ V_1 $ is first taken into account in order to compute $ S_1 $. In order to prevent double counting, we next deduct $ |U''(R)| $ since each vertex pair $(e_1, u_j)$ and $(e_1, u_j^{-1})$ was counted twice. As a result, now
			
			\begin{align*}
				S_1&=\displaystyle\sum\limits_{\substack{u_l\in U(R)\\j<l\leq |U(R)|}} \sum\limits_{\substack{u_j\in U(R)\\j=1}}^{|U(R)|-1}d((e_1,u_j),(e_1,u_l))-2\frac{|U''(R)|}{2}\\
				&=\sum\limits_{\substack{u_l\in U(R)\\j<l\leq |U(R)|}}\{d((e_i,u_1),(e_i,u_l))+d((e_i,u_2),(e_i,u_l))+\ldots+d((e_i,u_{|U(R)|-1}),(e_i,u_l ))\}-|U''(R)|\\
				&=3[(|U(R)|-1)+(|U(R)|-2)+\ldots+1]-|U''(R)|\\
				&=3\left(\frac{|U(R)|\left(|U(R)|-1\right)}{2}\right)-|U(R)|+|U'(R)|\\
				&=\frac{1}{2}\left[3|U(R)|^2-5|U(R)|+2|U'(R)|\right]
			\end{align*}
			\item $ \displaystyle S_2=\sum\limits_{\substack{e_i\in Id(R)^*\\i=2}}^{2^n-1}\sum\limits_{\substack{u_l\in U(R)\\j<l\leq|U(R)|}} \sum\limits_{\substack{u_j\in U(R)\\j=1}}^{|U(R)|-1}d((e_i,u_j),(e_i,u_l)) $
			
			By Lemma \ref{Lemma 0.1}, $d((e_i,u_j),(e_i,u_l))=\begin{cases}
				2 & \text{ if } e_i\neq1\text{ and }  u_ju_l\neq1\\
				1 & \text{ if } e_i=1\text{ and } u_ju_l=1.
			\end{cases}$
			
			To determine $S_2$, we use the same argument as in $S_1$. We initially consider the distance $d((e_i,u_j),(e_i,u_l))=2$ for each pair of vertices $(e_i,u_j)$ and $ (e_i,u_l) $, $u_ju_l\neq1$, in $V_i$. Then, we subtract $\frac{|U''(R)|}{2}$ to avoid double counting as each pair $(e_i,u_j)$ and $(e_i,u_j^{-1})$ was counted twice. As a result, now
			
			\begin{align*}
				S_2&=\sum\limits_{\substack{e_i\in Id(R)^*\\i=2}}^{2^n-1}\left[\sum\limits_{\substack{u_l\in U(R)\\j<l\\l=2}}^{|U(R)|} \sum\limits_{\substack{u_j\in U(R)\\j=1}}^{|U(R)|-1}d((e_i,u_j),(e_i,u_l))-\frac{|U''(R)|}{2}\right]\\
				&=\sum\limits_{\substack{j<l\\j=2}}^{2^n-1}\left[\sum\limits_{\substack{u_l\in U(R)\\j<l\\l=2}}^{|U(R)|}\{d((e_i,u_1),(e_i,u_l))+d((e_i,u_2),(e_i,u_l))+\ldots+d((e_i,u_{|U(R)|-1}),(e_i,u_l ))\}\right.\\
				&\left.\quad-\frac{|U''(R)|}{2}\right]\\
				&=(2^n-2)\left(2(|U(R)|-1)+2(|U(R)|-2)+\ldots+2\cdot1-\frac{|U''(R)|}{2}\right)\\
				&=(2^n-2)\left(2\left(\frac{|U(R)|\left(|U(R)|-1\right)}{2}\right)-\frac{|U(R)|-|U'(R)|}{2}\right)\\
				&=\frac{2^n-2}{2}\left(2|U(R)|^2-2|U(R)|-|U(R)|+|U'(R)|\right)\\
				&=\frac{2^n-2}{2}\left(2|U(R)|^2-3|U(R)|+|U'(R)|\right)
			\end{align*}

			\item 	$\displaystyle S_3=\sum\limits_{\substack{e_1,e_k\in Id(R)^*\\e_1\neq e_k}}\sum\limits_{\substack{u_l\in U(R)}}\sum\limits_{\substack{u_j\in U(R)}}d((e_1,u_j),(e_k,u_l)) $
			
			By Lemma \ref{Lemma 0.1}, we have $d((e_i,u_j),(e_k,u_l))=\begin{cases}
				2 & \text{ if } e_i\neq e_k, e_ie_k\neq0 \text{ and } u_ju_l\neq1\\
				1&  \text{ if }  e_ie_k=0 \text{ or } u_ju_l\neq1.
			\end{cases}$
			
			So, $S_3$ can be written as
			
			\begin{align*}
				S_3&=\sum\limits_{\substack{e_1,e_k\in Id(R)^*\\e_1\neq e_k\\k=2}}^{2^n-1}\sum\limits_{\substack{u_l\in U(R)\\u_lu_j\neq1}}\sum\limits_{\substack{u_j\in U(R)\\j=1}}^{|U(R)|}d((e_1,u_j),(e_k,u_l)) +\sum\limits_{\substack{e_1\neq e_k\\k=2}}^{2^n-1}\sum\limits_{\substack{u_j,u_l\in U(R)\\u_ju_l=1}}d((e_1,u_j),(e_k,u_l))\\
				&= 2(2^n-2)|U(R)|(|U(R)|-1)+(2^n-2)|U(R)|\\
				&=(2^n-2)|U(R)|(2|U(R)|-1)
			\end{align*}
			
			\item $\displaystyle S_4=\sum\limits_{\substack{e_i,e_k\in Id(R)^*\\e_i\neq e_k\neq e_1}}\sum\limits_{\substack{u_l\in U(R)}}\sum\limits_{\substack{u_j\in U(R)}}d((e_i,u_j),(e_k,u_l))$
			
			Given that $e_i\neq e_k\neq e_1$ and $ u_l,u_j\in U(R)$, we can divide $S_4$ into four possible parts to compute $S_4$:
			
			\begin{align*}
				S_4&= \sum\limits_{\substack{e_i,e_k\in Id(R)^*\\e_ie_k=0}}\sum\limits_{\substack{u_l\in U(R)}}\sum\limits_{\substack{u_j\in U(R)}}d((e_i,u_j),(e_k,u_l))+\sum\limits_{\substack{e_i,e_k\in Id(R)^*\\e_ie_k\neq0}}\sum\limits_{\substack{u_j,u_l\in U(R)\\u_ju_l=1}}d((e_i,u_j),(e_k,u_l))\\
				&\quad+ \sum\limits_{\substack{e_i,e_k\in Id(R)^*\\e_ie_k\neq0}}\sum\limits_{\substack{u_l\in U(R)\\u_ju_l\neq1}}\sum\limits_{\substack{u_j\in U(R)}}d((e_i,u_j),(e_k,u_l))
			\end{align*}
			
			Let us say 
			\begin{align*}
				T_1&=\sum\limits_{\substack{e_i,e_k\in Id(R)^*\\e_ie_k=0}}\sum\limits_{\substack{u_l\in U(R)}}\sum\limits_{\substack{u_j\in U(R)}}d((e_i,u_j),(e_k,u_l))\\
				T_2&= \sum\limits_{\substack{e_i,e_k\in Id(R)^*\\e_ie_k\neq0}}\sum\limits_{\substack{u_j,u_l\in U(R)\\u_ju_l=1}}d((e_i,u_j),(e_k,u_l))\\
				T_3&=\sum\limits_{\substack{e_i,e_k\in Id(R)^*\\e_ie_k\neq0}}\sum\limits_{\substack{u_l\in U(R)\\u_ju_l\neq1}}\sum\limits_{\substack{u_j\in U(R)}}d((e_i,u_j),(e_k,u_l))
			\end{align*}
			We calculate each summation separately,
			
			\begin{enumerate}
				\item $\displaystyle T_1=\sum\limits_{\substack{e_i,e_k\in Id(R)^*\\e_ie_k=0}}\sum\limits_{\substack{u_l\in U(R)}}\sum\limits_{\substack{u_j\in U(R)}}d((e_i,u_j),(e_k,u_l))$
				
				With the same reasoning as in the previous case in $T_4$, 
				
				\begin{align*}
					T_1&=\sum\limits_{\substack{e_i,e_k\in Id(R)^*\\e_ie_k=0}}\sum\limits_{\substack{u_l\in U(R)}}\sum\limits_{\substack{u_j\in U(R)}}d((e_i,u_j),(e_k,u_l))\\
					&=\left(\frac{2^n-2}{2}\right)|U(R)|^2
				\end{align*}
				

				\item 	$T_2=\displaystyle \sum\limits_{\substack{e_i,e_k\in Id(R)^*\\e_ie_k\neq0}}\sum\limits_{\substack{u_j,u_l\in U(R)\\u_ju_l=1}}d((e_i,u_j),(e_k,u_l))$
				
				Note that, for each pair of vertices $(e_i,u_j)\in V_i$ and $(e_i-1,u_j)\in V_{i+1}$ for each $i=2m$ and $m=1,2,\cdots,\frac{2^n-4}{2}$ there exists $(e_k,u_l)=(e_k,u_j^{-1})\in V_k$, $k>2m+1$, such that $d((e_i,u_j),(e_k,u_l))=d((e_i-1,u_j),(e_k,u_l))=1$. Therefore, $T_2$ can be written as 
				
				\begin{align*}
					T_2&=2\sum\limits_{\substack{e_k\in Id(R)^*\\k>2m+1}}\sum\limits_{\substack{e_i\in Id(R)^*\\i=2m\\m=1}}^{\frac{2^n-4}{2}}\sum\limits_{\substack{u_j,u_j^{-1}\in U(R)}}d((e_{2m},u_j),(e_k,u_j^{-1}))\\
					&=2((2^n-4)+(2^n-6)+\cdots+2)|U(R)|\\
					&=4\left(\left(\frac{2^n-4}{2}\right)+\left(\frac{2^n-6}{2}\right)+\cdots+1\right)|U(R)|\\
					&=4\left(\frac{\left(\frac{2^n-4}{2}\right)\left(\frac{2^n-4}{2}+1\right)}{2}\right)|U(R)|\\
					&=\frac{1}{2}(2^n-4)(2^n-2)|U(R)|.
				\end{align*}			
				
				\item $\displaystyle T_3=\sum\limits_{\substack{e_i,e_k\in Id(R)^*\\e_ie_k\neq0}}\sum\limits_{\substack{u_l\in U(R)\\u_ju_l\neq1}}\sum\limits_{\substack{u_j\in U(R)}}d((e_i,u_j),(e_k,u_l))$
				
				By Lemma \ref{Lemma 0.1}, we have $d((e_i,u_j,(e_k,u_l))=2$ if $e_i\neq e
				_k\neq e_1$, $e_i e_k\neq0$ and $u_ju_l\neq1$. Since $d((e_i,u_j,(e_k,u_l))=d((e_i-1,u_j,(e_k,u_l))=2$ as $e_i(e_i-1)=0$. So, using the same argument as in the previous case in $S_2$, we compute $T_3$.
				
				\begin{align*}
					T_3&=2\sum\limits_{\substack{e_k\in Id(R)^*\\2m+1<k\leq 2^n-1}}\sum\limits_{\substack{e_i\in Id(R)*\\i=2m\\m=1}}^{\frac{2^n-4}{2}}\sum\limits_{\substack{u_l\in U(R)\\u_ju_l\neq 1}}\sum\limits_{\substack{u_j\in U(R)\\j=1}}^{|U(R)|}d((e_i,u_j,(e_k,u_l))\\
					&= 2(|U(R)|-1)(|U(R)|)(2(2^n-4)+2(2^n-6)+\cdots+2.2)\\
					&=8(|U(R)|-1)(|U(R)|)\left(\frac{2^n-4}{2}+\frac{2^n-6}{2}+\cdots+1\right)\\
					&=8(|U(R)|-1)(|U(R)|)\left(\frac{\left(\frac{2^n-4}{2}\right)\left(\frac{2^n-4}{2}+1\right)}{2}\right)\\
					&=(|U(R)|-1)(|U(R)|)(2^n-4)(2^n-2).
				\end{align*}
			\end{enumerate}		
			Combining $T_1, T_2$, and $T_3$ produces $S_4$ as
			\begin{align*}
				S_4&=\frac{1}{2} \left(2^n-2\right)|U(R)|^2+\frac{1}{2}(2^n-4)(2^n-2)|U(R)|+(|U(R)|-1)(|U(R)|)(2^n-4)(2^n-2).
			\end{align*}
			Simplifying further yields
			
			\begin{align*}
				S_4=\frac{1}{2}[|U(R)|^2(2\cdot2^{2n}-11\cdot2^n+14)-|U(R)|(2^{2n}-6\cdot2^n+8)].
			\end{align*}
		\end{enumerate}	
		
		Finally, to find the Wiener index $W(Cl_2(R))$ of clean graph $Cl_2(R)$ of ring $R$, we add $S_1$, $S_2$, $S_3$ and $S_4$. Thus
		
		\begin{align*}
			W(Cl_2(R))&=\frac{1}{2}\left[3|U(R)|^2-5|U(R)|+2|U'(R)|\right]+\frac{1}{2}(2^n-2)\left(2|U(R)|^2-3|U(R)|+|U'(R)|\right)\\
			&\quad+(2^n-2)|U(R)|(2|U(R)|-1)+\frac{1}{2}[|U(R)|^2(2\cdot2^{2n}-11.2^n+14)-|U(R)|(2^{2n}-6.2^n+8)].
		\end{align*}  
		
		Simplifying further, we obtain 
		
		\begin{align*}
			W(Cl_2(R))=\frac{1}{2}\left[|U(R)|^2(2\cdot2^{2n}-5\cdot2^n+5)-|U(R)|(2^{2n}-2^n+3)+|U'(R)|2^n\right].
		\end{align*}
	\end{proof}
	
	Now, we determine the Wiener index of ring $\mathbb Z_n$, the ring of integers modulo $n$. We need the following proposition to find the Wiener index of $\mathbb Z_n$.
	\begin{proposition}\label{Prop 3} For any non negative integer $m$, $n=2^m\prod\limits_{i=1}^{h}{p_i}^{\alpha_i}$, $p_i\neq2$ and $p_i$'s are distinct primes dividing $n$, in ring $\mathbb{Z}_n$,
		\begin{center}
			$|U'(\mathbb{Z}_n)|=\begin{cases}
				2^h & \text{ if }m=0 \text{ or }m=1 \\
				2^{h+1} & \text{ if } m=2\\
				2^{h+2} & \text{ if } m\geq 3.
			\end{cases}$
		\end{center}
		Note that, for our convenience, we write $|U'(\mathbb{Z}_n)|=r$, and the value of $r$ changes accordingly.
		
	\end{proposition}
	\begin{proof}
		
		To find the number of self-invertible elements \( |U'(\mathbb{Z}_n)| \) in \( \mathbb{Z}_n \), we determine the solution for \( x^2 \equiv 1 \pmod{n} \), where \( n = 2^m \prod_{i=1}^{h}{p_i}^{\alpha_i} \), \( p_i \neq 2 \). The equation \( x^2 \equiv 1 \pmod{n} \) implies \( x^2 \equiv 1 \pmod{2^m} \) and \( x^2 \equiv 1 \pmod{\prod_{i=1}^{h}{p_i}^{\alpha_i}} \).                
		
		The solutions of $x^2\equiv1(mod~2^m)$ are as follows:
		\begin{enumerate}
			\item  $x\equiv1(mod~2)$, for $m=1$.
			\item $x\equiv1(mod~4)$ and $x\equiv3(mod~4)$, for $m=2$.
			
			\item 	By \cite[Theorem 4.20]{Jacobson}, 
			and \cite[Article 90 - 91]{Gauss} 
			$x^2\equiv1(mod~2^m)$, for $m\geq3$, has four distinct solutions $\{1, 2^{m-1}\pm 1, 2^m-1\}$ in multiplicative group $U_{2^m}$ of units of $\mathbb{Z}_{2^m}$.
		\end{enumerate}
		The solutions of $x^2\equiv1(mod ~{p_i}^{\alpha_i})$ are $x\equiv1(mod ~{p_i}^{\alpha_i})$ and $x\equiv ({p_i}^{\alpha_i}-1)(mod ~{p_i}^{\alpha_i})$.
		Now, by the Chinese remainder theorem,
		
		\begin{center}
			$|U'(\mathbb{Z}_n)|=\begin{cases}
				2^h & \text{ if }m=0 \text{ or }m=1 \\
				2^{h+1} & m=2\\
				2^{h+2} & m\geq 3.
			\end{cases}$
		\end{center}
		
	\end{proof}	
	
	\begin{corollary}
		Let $n=\prod_{i=1}^{k}{p_i}^{\alpha_i}$, $\alpha_i \in \mathbb{N}$ for $i=1,2,...,k$ and $p_i$'s be distinct primes. Then the Wiener index $W(Cl_2(\mathbb{Z}_n))$ of the clean graph $Cl_2(\mathbb{Z}_n)$ of ring of integers $\mathbb{Z}_n$ modulo $n$ is defined by
		\begin{align*}
			\displaystyle W(Cl_2(\mathbb{Z}_n))=\begin{cases}
				\infty & \text{ if } k=1,\\
				\frac{1}{2}\left[\phi(n)^2(2\cdot2^{2k}-5\cdot2^k+5)-\phi(n)(2^{2k}-2^k+3)+2^kr\right] & \text{ if } k\geq2 \text{ and }r\neq 0.
			\end{cases}
		\end{align*}
	\end{corollary}
	\begin{Illustration}
		Figure \ref{Figure 2} illustrates the above result. In this graph, we have partitioned the vertex set $ V(Cl_2(\mathbb{Z}_{pq})) $ into three disjoint subsets as $\displaystyle V_1= \{(1,u):1\in Id(\mathbb{Z}_{pq}), u\in U(\mathbb{Z}_{pq})\}$, $ V_2= \{(e,u):e\in Id(\mathbb{Z}_{pq}), u\in U(\mathbb{Z}_{pq})\}$ and $ V_3= \{(1-e,u):1-e\in Id(\mathbb{Z}_{pq}), u\in U(\mathbb{Z}_{pq})\}$.   
	\end{Illustration}
	
	\begin{example}
		In the ring $\mathbb{Z}_{15}$, the set of units $U(\mathbb{Z}_{15})$ consists of the elements $\{1, 2, 4, 7, 8, 11, 13, 14\}$, and the set of idempotents $Id(\mathbb{Z}_{15})^*$ is $\{1, 6, 10\}$. Also, the set $V(Cl_2(\mathbb{Z}_{15}))$ consists of pairs like $(1, 1)$, $(1, 2)$, ..., $(1, 14)$, $(6, 1)$, $(6, 2)$, ..., $(6, 14)$, ..., and $(10, 14)$. Here $ n=15 $, $ \phi(15)= 8 $, $ k=2 $ and $r=4$. So that, $ W(Cl_2(\mathbb{Z}_{15}))=\frac{1}{2}[8^2(2(2^{4})-5(2^2)+5)-8(2^4-2^2+3)+2^22^2] =332$.
	\end{example}		
	
	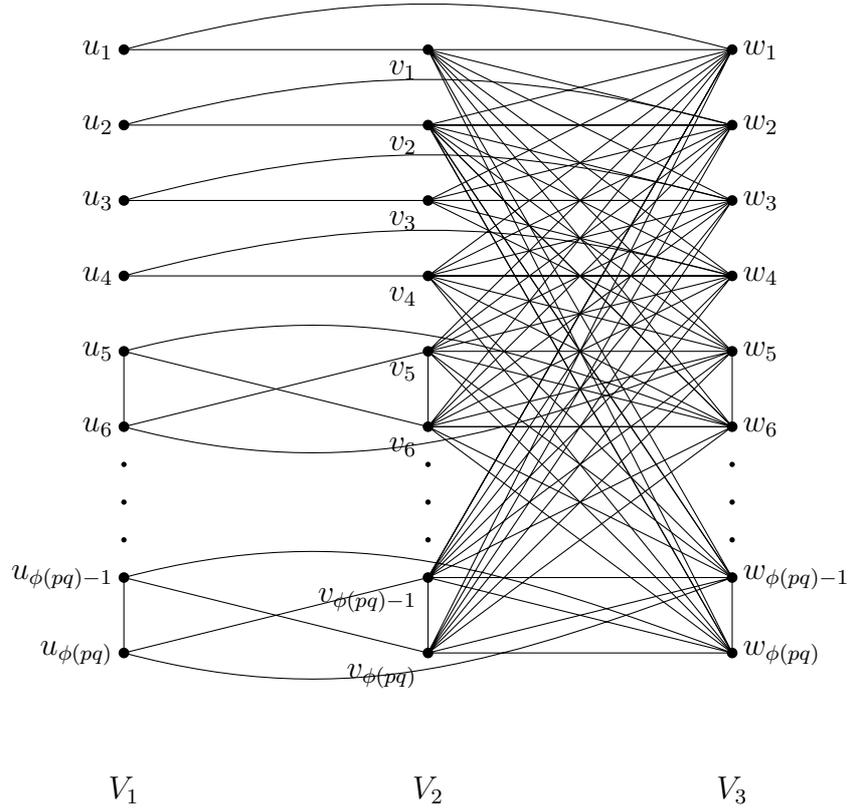
\begin{figure}[h!]
		\begin{center}
			
			{	\begin{tikzpicture}
					\draw (0,-6.5) node[below]{$V_3$};
					\fill (0,3) circle (2pt) node[right] {$w_1$}; 
					\fill (0,2) circle (2pt) node[right] {$w_2$}; 
					\fill (0,1) circle (2pt) node[right] {$w_3$}; 
					\fill (0,0) circle (2pt) node[right] {$w_4$};
					\fill (0,-1) circle (2pt) node[right] {$w_5$};
					\fill (0,-2) circle (2pt) node[right] {$w_6$};
					\fill (0,-2.5) circle (1pt); 
					\fill (0,-3) circle (1pt);
					\fill (0,-3.5) circle (1pt);
					\fill (0,-4) circle (2pt) node[right] {$w_{\phi(pq)-1}$};
					\fill (0,-5) circle (2pt) node[right] {$w_{\phi(pq)}$};
					
					\draw (-4,-6.5) node[below]{$V_2$};
					\fill (-4,3) circle (2pt) node[below left] {$v_1$}; 
					\fill (-4,2) circle (2pt) node[below left] {$v_2$}; 
					\fill (-4,1) circle (2pt) node[below left] {$v_3$}; 
					\fill (-4,0) circle (2pt) node[below left] {$v_4$};
					\fill (-4,-1) circle (2pt) node[below left] {$v_5$};
					\fill (-4,-2) circle (2pt) node[below left] {$v_6$};
					\fill (-4,-2.5) circle (1pt); 
					\fill (-4,-3) circle (1pt);
					\fill (-4,-3.5) circle (1pt);
					\fill (-4,-4) circle (2pt) node[below left] {$v_{\phi(pq)-1}$};
					\fill (-4,-5) circle (2pt) node[below left] {$v_{\phi(pq)}$};
					
					\draw (-8,-6.5) node[below]{$V_1$};
					\fill (-8,3) circle (2pt) node[left] {$u_1$}; 
					\fill (-8,2) circle (2pt) node[left] {$u_2$}; 
					\fill (-8,1) circle (2pt) node[left] {$u_3$}; 
					\fill (-8,0) circle (2pt) node[left] {$u_4$};
					\fill (-8,-1) circle (2pt) node[left] {$u_5$};
					\fill (-8,-2) circle (2pt) node[left] {$u_6$};
					\fill (-8,-2.5) circle (1pt); 
					\fill (-8,-3) circle (1pt);
					\fill (-8,-3.5) circle (1pt);
					\fill (-8,-4) circle (2pt) node[left] {$u_{\phi(pq)-1}$};
					\fill (-8,-5) circle (2pt) node[left] {$u_{\phi(pq)}$};
					
					\draw [black] (-8,3) to (-4,3) to (0,3) (-8,3) to[out=15,in=165] (0,3) (-4,3) to (0,2) (-4,3) to (0,1) (-4,3) to (0,0) (-4,3) to (0,-1) (-4,3) to (0,-2) (-4,3) to (0,-4) (-4,3) to (0,-5) ;
					\draw [black] (-8,2) to (-4,2) to (0,2) (-8,2) to[out=15,in=165] (0,2) (-4,2) to (0,1) (-4,2) to (0,2)(-4,2) to (0,3) (-4,2) to (0,0) (-4,2) to (0,-1) (-4,2) to (0,-2) (-4,2) to (0,-4) (-4,2) to (0,-5) ;
					\draw [black] (-8,1) to (-4,1) to (0,1) (-8,1) to[out=15,in=165] (0,1) (-4,1) to (0,1) (-4,1) to (0,2)(-4,1) to (0,3) (-4,1) to (0,0) (-4,1) to (0,-1) (-4,2) to (0,-2) (-4,2) to (0,-4) (-4,2) to (0,-5) ;
					\draw [black] (-8,0) to (-4,0) to (0,0) (-8,0) to[out=15,in=165] (0,0) (-4,0) to (0,0) (-4,0) to (0,2)(-4,0) to (0,3) (-4,0) to (0,1) (-4,0) to (0,-1) (-4,0)to (0,-2) (-4,0) to (0,-4) (-4,0) to (0,-5) ;
					\draw [black] (-8,-2) to (-8,-1)to (-4,-2) to (0,-2) (-8,-1) to[out=15,in=160] (0,-2) (-4,-1)to (-4,-2) (0,-1)to(0,-2) (-4,-1) to (0,0) (-4,-1) to (0,2)(-4,-1) to (0,3) (-4,-1) to (0,1) (-4,-1) to (0,-1) (-4,-1) to (0,-2) (-4,-1) to (0,-4) (-4,-1) to (0,-5);
					\draw[black] (-8,-2) to (-4,-1) (-8,-2) to [out=-15,in=-160](0,-1) (-4,-2)to(0,3) (-4,-2)to(0,2) (-4,-2) to (0,1) (-4,-2) to (0,0) (-4,-2) to (0,-1) (-4,-2) to (0,-2) (-4,-2) to (0,-4) (-4,-2) to (0,-5);
					\draw[black] (-8,-5) to (-8,-4)to(-4,-5)to(-4,-4)to(0,3) (-8,-4) to[out=15,in=160] (0,-5) (-8,-5)to[out=-15,in=-160](0,-4)to(0,-5)(-8,-5)to(-4,-4)to(0,3)(-4,-4)to(0,2)(-4,-4)to(0,1)(-4,-4)to(0,0)(-4,-4)to(0,-1)(-4,-4)to(0,-2)(-4,-4)to(0,-4)(-4,-4)to(0,-5) (-4,-5)to(0,3)(-4,-5)to(0,2)(-4,-5)to(0,1)(-4,-5)to(0,0)(-4,-5)to(0,-1)(-4,-5)to(0,-2)(-4,-5)to(0,-4)(-4,-5)to(0,-5);
			\end{tikzpicture}}
		\end{center}
		\caption{$Cl_2{(\mathbb{Z}_{pq})}$}\label{Figure 2}
	\end{figure}

\section{Matchings in $Cl_2(R)$}\label{Sec4}

In this section, we show $Cl_2(R)$ has a perfect matching. Also, we determine the matching number if it does not contain a perfect matching.
\begin{theorem}
	Let $R$ be any commutative Artinian ring with $|U(R)|=2k$ for some positive integer $k$. Then, $Cl_2(R)$ contains a perfect matching. Moreover, if $|U(R)|$ is not even then the matching number is $\left(\frac{|U(R)|(2^n-1)-1}{2}\right)$.
\end{theorem}

\begin{proof} Given that $R$ is a commutative Artinian ring. By \cite[Theorem 8.7]{Atiyah}, $R\cong R_1\times R_2\times\cdots\times R_n$, where each $R_i$ is a local ring. Let us partition the vertex set $V(Cl_2(R))$ as $V_1\sqcup V_2\sqcup...\sqcup V_{2^n-1},$ where $V_{i}=\{v_{ji}=(e_i,u_j): e_i\in Id(R)^*, u_j\in U(R)\}$ as shown in the Figure \ref{Figure 1}. For convenience, we write $|U'(R)|=r$.

	By the definition of the clean graph, we have the following adjacency:
	\begin{enumerate}
		\item $(e_i,u_j)\sim (e_k,u_j)$ if $u_j \in U'(R)$ or $e_ie_k=0.$
		\item  $(e_i,u_j)\sim (e_i,u_l)$ if $u_j,u_l\in U''(R)$ and $u_ju_l=1$.
		\item $(e_i,u_j)\sim (e_k,u_l) $ if $u_j,u_l\in U''(R)$ and $u_ju_l=1$ or $e_i e_k= 0$ .
		
	\end{enumerate}

	For every pair of units $u_i$ and $u_j$ in $U(R)$, there is a subgraph that contains a perfect matching. As for units $u_i,u_j\in U'(R)$ where $i\neq j$, one of the matching in corresponding subgraph is given by 
	\begin{align*}
		M_{ij}&=\{(e_{2k},u_i)(e_{2k+1},u_i):u_i\in U'(R),e_{2k},e_{2k+1}\in Id(R)^*, k=1,2,\cdots,\frac{(2^n-4)}{2}\}\\
		&\quad\cup\{(e_{2k-1},u_j)(e_{2k},u_j):u_j\in U'(R),e_{2k-1},e_{2k}\in Id(R)^*, k=1,2,\cdots,\frac{(2^n-2)}{2}\}\\
		&\quad\cup \{(e_1,u_i)(e_{2^n-1},u_i)\}\cup \{(e_{2^n-2},u_i)(e_{2^n-1},u_j)\}.
	\end{align*}
	The matching number of the subgraph corresponding to $u_i$ and $u_j$ is :
	
	\begin{align*}
		|M_{ij}|&= \frac{2^n-4}{2}+\frac{2^n-2}{2}+2\\
		&= 2^n-1.
	\end{align*}
	
	Since the matching number $M_{ij}$ in the corresponding subgraph is half of the number of the vertices in the subgraph. Therefore, $M_{ij}$ is a perfect matching in the corresponding subgraph. Since $|U'(R)|$ is even, there exist $\frac{|U'(R)|}{2}$ vertex disjoint subgraphs. Let us define $ M_1 = \{M_{ij} : i \neq j\} $. As we are considering disjoint pairs of units $ u_i$ and $u_j \in U'(R)$, any $M_{ij} $ and $M_{kl}$ in  $M_1$ are vertex-disjoint. Similarly, for $u_p,u_q\in U''(R)$ where $u_q=u_p^{-1}$, one of the perfect matching in corresponding subgraph is $M_{pq}=\{(e_k,u_p)(e_k,u_q): u_p,u_q\in U''(R), u_pu_q=1, \text{ and } e_k\in Id(R)^*, k=1,2,\cdots , 2^n-1\}$. Let us define $M_2=\{M_{pq}:p\neq q\}$. Since for each $u_p\in U''(R)$ there exists unique $u_q\in U''(R)$, any $M_{pq}$ and $M_{mn}$ in $M_2$ are vertex disjoint. So, the matching in $Cl_2(R)$ is $M=M_1\cup M_2$. The matching number $\mu(Cl_2(R))$ in the clean graph $Cl_2(R)$ is determined as:
	
	\begin{align*}
		\mu(Cl_2(R))&= |M|\\
		&=|M_1|+|M_2|\\
		&= \frac{|U'(R)|(2^n-1)}{2}+\frac{|U''(R)|(2^n-1)}{2}\\
		&=\frac{(2^n-1)}{2}(|U'(R)|+|U''(R)|)\\
		&=\frac{(2^n-1)|U(R)|}{2}
	\end{align*}
	
	Since $|V(Cl_2(R)|=(2^n-1)|U(R)|$ and $\mu(Cl_2(R))=\frac{(2^n-1)|U(R)|}{2}$, therefore, the clean graph $Cl_2(R)$ of a ring $R$ contains a perfect matching.

	Since $U(R)=U'(R)\cup U''(R)$. Therefore, if $|U(R)|\neq2k$ for some positive integer $k$, then $|U'(R)|$ is odd as $U''(R)=\{u_j,u_l\in R : u_ju_l=1, u_j\neq u_l\}$. So, there exists one $u\in U'(R)$ that does not form a pair. The matching in subgraph correspond to $u$ is $\{(e_{2i-1},u)(e_{2i},u): u \in U'(R), e_{2i-1},e_{2i}\in Id(R)^* \text{ and }i=1,2,\cdots,\frac{2^n-2}{2}\}$. It is easy to observe that only one vertex $(e_{2^n-1},u)$ remains unsaturated. So, the matching number is $\left(\frac{|U(R)|(2^n-1)-1}{2}\right)$.   
\end{proof}




The following illustration will help to understand the above proof of perfect matching clean graph $Cl_2(R)$ of a ring $R$.

\begin{Illustration}
	Let $R$ be a commutative Artinian ring with unity with $|U(R)|=2k$ for some positive integer $k$. We partition the vertex set $V(Cl_2(R))$ of the clean graph $Cl_2(R)$. We partition $Cl_2(R)$; see \ref{remark 1}. Now, in Figure \ref{Figure 3}, we draw a perfect matching only. We can see all the vertices are saturated. Thus, $Cl_2(R)$ has a perfect matching.
	
	\begin{figure}[h!]
		\begin{center}
			\scalebox{0.8}{\begin{tikzpicture}
					
					\draw (-8,-7) node[below]{$V_1$};
					\fill (-8,7) circle (3pt) node[below left] {$v_{11}$}; 
					\fill (-8,6) circle (3pt) node[below left] {$v_{21}$}; 
					\fill (-8,5) circle (3pt) node[below left] {$v_{31}$}; 
					\fill (-8,4) circle (3pt) node[below left] {$v_{41}$};
					\fill (-8,3.5) circle (1pt); 
					\fill (-8,3) circle (1pt); 
					\fill (-8,2.5) circle (1pt);
					\fill (-8,2) circle (3pt) node[below left] {$v_{(r-1)1}$};
					\fill (-8,1) circle (3pt) node[below left] {$v_{r1}$};
					\fill (-8,0) circle (3pt) node[below left] {$v_{(r+1)1}$};
					\fill (-8,-1) circle (3pt) node[below left] {$v_{(r+2)1}$};
					
					\fill (-8,-2) circle (3pt) node[below left] {$v_{(r+3)1}$};
					\fill (-8,-3) circle (3pt) node[below left] {$v_{(r+4)1}$};
					
					\fill (-8,-3.5) circle (1pt);
					\fill (-8,-4) circle (1pt);
					\fill (-8,-4.5) circle (1pt);
					
					\fill (-8,-5) circle (3pt) node[below left] {$v_{(|U(R)|-1)1}$};
					\fill (-8,-6) circle (3pt) node[below left] {$v_{|U(R)|1}$};
					
					\draw (-5.5,-7) node[below]{$V_2$};
					\fill (-5.5,7) circle (3pt) node[below left] {$v_{12}$}; 
					\fill (-5.5,6) circle (3pt) node[below left] {$v_{22}$}; 
					\fill (-5.5,5) circle (3pt) node[below left] {$v_{32}$}; 
					\fill (-5.5,4) circle (3pt) node[below left] {$v_{42}$};
					\fill (-5.5,3.5) circle (1pt); 
					\fill (-5.5,3) circle (1pt); 
					\fill (-5.5,2.5) circle (1pt);
					\fill (-5.5,2) circle (3pt) node[below left] {$v_{(r-1)2}$};
					\fill (-5.5,1) circle (3pt) node[below left] {$v_{r2}$};
					\fill (-5.5,0) circle (3pt) node[below left] {$v_{(r+1)2}$};
					\fill (-5.5,-1) circle (3pt) node[below left] {$v_{(r+2)2}$};
					
					\fill (-5.5,-2) circle (3pt) node[below left] {$v_{(r+3)2}$};
					\fill (-5.5,-3) circle (3pt) node[below left] {$v_{(r+4)2}$};
					
					\fill (-5.5,-3.5) circle (1pt);
					\fill (-5.5,-4) circle (1pt);
					\fill (-5.5,-4.5) circle (1pt);
					
					\fill (-5.5,-5) circle (3pt) node[below left] {$v_{(|U(R)|-1)2}$};
					\fill (-5.5,-6) circle (3pt) node[below left] {$v_{|U(R)|2}$};

					
					\draw (-2.5,-7) node[below]{$V_3$};
					\fill (-2.5,7) circle (3pt) node[below left] {$v_{13}$}; 
					\fill (-2.5,6) circle (3pt) node[below left] {$v_{23}$}; 
					\fill (-2.5,5) circle (3pt) node[below left] {$v_{33}$}; 
					\fill (-2.5,4) circle (3pt) node[below left] {$v_{43}$};
					\fill (-2.5,3.5) circle (1pt); 
					\fill (-2.5,3) circle (1pt); 
					\fill (-2.5,2.5) circle (1pt);
					\fill (-2.5,2) circle (3pt) node[below left] {$v_{(r-1)3}$};
					\fill (-2.5,1) circle (3pt) node[below left] {$v_{r3}$};
					\fill (-2.5,0) circle (3pt) node[below left] {$v_{(r+1)3}$};
					\fill (-2.5,-1) circle (3pt) node[below left] {$v_{(r+2)3}$};
					
					\fill (-2.5,-2) circle (3pt) node[below left] {$v_{(r+3)3}$};
					\fill (-2.5,-3) circle (3pt) node[below left] {$v_{(r+4)3}$};
					
					\fill (-2.5,-3.5) circle (1pt);
					\fill (-2.5,-4) circle (1pt);
					\fill (-2.5,-4.5) circle (1pt);
					
					\fill (-2.5,-5) circle (3pt) node[below left] {$v_{(|U(R)|-1)3}$};
					\fill (-2.5,-6) circle (3pt) node[below left] {$v_{|U(R)|3}$};
					
					\draw (0,-7) node[below]{$V_4$};
					\fill (0,7) circle (3pt) node[below left] {$v_{14}$}; 
					\fill (0,6) circle (3pt) node[below left] {$v_{24}$}; 
					\fill (0,5) circle (3pt) node[below left] {$v_{34}$}; 
					\fill (0,4) circle (3pt) node[below left] {$v_{44}$};
					\fill (0,3.5) circle (1pt); 
					\fill (0,3) circle (1pt); 
					\fill (0,2.5) circle (1pt);
					\fill (0,2) circle (3pt) node[below left] {$v_{(r-1)4}$};
					\fill (0,1) circle (3pt) node[below left] {$v_{r4}$};
					\fill (0,0) circle (3pt) node[below left] {$v_{(r+1)4}$};
					\fill (0,-1) circle (3pt) node[below left] {$v_{(r+2)4}$};
					
					\fill (0,-2) circle (3pt) node[below left] {$v_{(r+3)4}$};
					\fill (0,-3) circle (3pt) node[below left] {$v_{(r+4)4}$};
					
					\fill (0,-3.5) circle (1pt);
					\fill (0,-4) circle (1pt);
					\fill (0,-4.5) circle (1pt);
					
					\fill (0,-5) circle (3pt) node[below left] {$v_{(|U(R)|-1)4}$};
					\fill (0,-6) circle (3pt) node[below left] {$v_{|U(R)|4}$};

					\fill  (1,7) circle (1pt)(1.3,7) circle (1pt)(1.6,7)  circle (1pt);
					\fill  (1,6) circle (1pt)(1.3,6) circle (1pt)(1.6,6)  circle (1pt);
					\fill  (1,5) circle (1pt)(1.3,5) circle (1pt)(1.6,5)  circle (1pt);
					\fill  (1,4) circle (1pt)(1.3,4) circle (1pt)(1.6,4)  circle (1pt);
					\fill  (1,2) circle (1pt)(1.3,2) circle (1pt)(1.6,2)  circle (1pt);
					\fill  (1,1) circle (1pt)(1.3,1) circle (1pt)(1.6,1)  circle (1pt);
					\fill  (1,0) circle (1pt)(1.3,0) circle (1pt)(1.6,0)  circle (1pt);
					\fill  (1,-1) circle (1pt)(1.3,-1) circle (1pt)(1.6,-1)  circle (1pt);
					\fill  (1,-2) circle (1pt)(1.3,-2) circle (1pt)(1.6,-2)  circle (1pt);
					\fill  (1,-3) circle (1pt)(1.3,-3) circle (1pt)(1.6,-3)  circle (1pt);
					\fill  (1,-5) circle (1pt)(1.3,-5) circle (1pt)(1.6,-5)  circle (1pt);
					\fill  (1,-6) circle (1pt)(1.3,-6) circle (1pt)(1.6,-6)  circle (1pt);
					
					
					\draw (3,-7) node[below]{$V_{2^n-3}$};
					\fill (3,7) circle (3pt) node[below left] {$v_{1{(2^n-3)}}$}; 
					\fill (3,6) circle (3pt) node[below left] {$v_{2{(2^n-3)}}$}; 
					\fill (3,5) circle (3pt) node[below left] {$v_{3{(2^n-3)}}$}; 
					\fill (3,4) circle (3pt) node[below left] {$v_{4{(2^n-3)}}$};
					\fill (3,3.5) circle (1pt); 
					\fill (3,3) circle (1pt); 
					\fill (3,2.5) circle (1pt);
					\fill (3,2) circle (3pt) node[below left] {$v_{(r-1){(2^n-3)}}$};
					\fill (3,1) circle (3pt) node[below left] {$v_{r{(2^n-3)}}$};
					\fill (3,0) circle (3pt) node[below left] {$v_{(r+1){(2^n-3)}}$};
					\fill (3,-1) circle (3pt) node[below left] {$v_{(r+2){(2^n-3)}}$};
					
					\fill (3,-2) circle (3pt) node[below left] {$v_{(r+3){(2^n-3)}}$};
					\fill (3,-3) circle (3pt) node[below left] {$v_{(r+4){(2^n-3)}}$};
					
					\fill (3,-3.5) circle (1pt);
					\fill (3,-4) circle (1pt);
					\fill (3,-4.5) circle (1pt);
					
					\fill (3,-5) circle (3pt) node[below left] {$v_{(|U(R)|-1){(2^n-3)}}$};
					\fill (3,-6) circle (3pt) node[below left] {$v_{|U(R)|{(2^n-3)}}$};
					
					
					\draw (6,-7) node[below]{$V_{2^n-2}$};
					\fill (6,7) circle (3pt) node[below left] {$v_{1{(2^n-2)}}$}; 
					\fill (6,6) circle (3pt) node[below left] {$v_{2{(2^n-2)}}$}; 
					\fill (6,5) circle (3pt) node[below left] {$v_{3{(2^n-2)}}$}; 
					\fill (6,4) circle (3pt) node[below left] {$v_{4{(2^n-2)}}$};
					\fill (6,3.5) circle (1pt); 
					\fill (6,3) circle (1pt); 
					\fill (6,2.5) circle (1pt);
					\fill (6,2) circle (3pt) node[below left] {$v_{(r-1){(2^n-2)}}$};
					\fill (6,1) circle (3pt) node[below left] {$v_{r{(2^n-2)}}$};
					\fill (6,0) circle (3pt) node[below left] {$v_{(r+1){(2^n-2)}}$};
					\fill (6,-1) circle (3pt) node[below left] {$v_{(r+2){(2^n-2)}}$};
					\fill (6,-2) circle (3pt) node[below left] {$v_{(r+3){(2^n-2)}}$};
					\fill (6,-3) circle (3pt) node[below left] {$v_{(r+4){(2^n-2)}}$};
					
					\fill (6,-3.5) circle (1pt);
					\fill (6,-4) circle (1pt);
					\fill (6,-4.5) circle (1pt);
					
					\fill (6,-5) circle (3pt) node[below left] {$v_{(|U(R)|-1){(2^n-2)}}$};
					\fill (6,-6) circle (3pt) node[below left] {$v_{|U(R)|{(2^n-2)}}$};
					
					\draw (9,-7) node[below]{$V_{2^n-1}$};
					\fill (9,7) circle (3pt) node[below left] {$v_{1{(2^n-1)}}$}; 
					\fill (9,6) circle (3pt) node[below left] {$v_{2{(2^n-1)}}$}; 
					\fill (9,5) circle (3pt) node[below left] {$v_{3{(2^n-1)}}$}; 
					\fill (9,4) circle (3pt) node[below left] {$v_{4{(2^n-1)}}$};
					\fill (9,3.5) circle (1pt); 
					\fill (9,3) circle (1pt); 
					\fill (9,2.5) circle (1pt);
					\fill (9,2) circle (3pt) node[below left] {$v_{(r-1){(2^n-1)}}$};
					\fill (9,1) circle (3pt) node[below left] {$v_{r{(2^n-1)}}$};
					\fill (9,0) circle (3pt) node[below left] {$v_{(r+1){(2^n-1)}}$};
					\fill (9,-1) circle (3pt) node[below left] {$v_{(r+2){(2^n-1)}}$};
					\fill (9,-2) circle (3pt) node[below left] {$v_{(r+3){(2^n-1)}}$};
					\fill (9,-3) circle (3pt) node[below left] {$v_{(r+4){(2^n-1)}}$};
					
					\fill (9,-3.5) circle (1pt);
					\fill (9,-4) circle (1pt);
					\fill (9,-4.5) circle (1pt);
					
					\fill (9,-5) circle (3pt) node[below left] {$v_{(|U(R)|-1){(2^n-1)}}$};
					\fill (9,-6) circle (3pt) node[below left] {$v_{|U(R)|{(2^n-1)}}$};
					\draw [line width=0.5mm] (-8,7) to [out=10,in=170](9,7); 
					\draw [line width=0.5mm] (-5.5,7) to (-2.5,7); 
					\draw [line width=0.5mm](0,7) to (0.5,7); 
					\draw [line width=0.5mm] (2.5,7) to (3,7); 
					\draw [line width=0.5mm](6,7) to (9,6); 
					\draw [line width=0.5mm] (-8,6) to (-5.5,6); 
					\draw[line width=0.5mm] (-2.5,6) to (0,6); 
					\draw [line width=0.5mm](3,6) to (6,6);
					\draw [line width=0.5mm] (-8,5)to [out=5,in=175](9,5); 
					\draw [line width=0.5mm] (-5.5,5)to (-2.5,5); 
					\draw [line width=0.5mm](0,5)to (0.5,5); 
					\draw [line width=0.5mm] (2.5,5)to (3,5); 
					\draw [line width=0.5mm] (-8,4) to (-5.5,4); 
					\draw[line width=0.5mm] (-2.5,4) to (0,4); 
					\draw [line width=0.5mm](3,4) to (6,4);
					\draw [line width=0.5mm] (-8,2) to [out=5,in=175](9,2); 
					
					\draw[line width=0.5mm] (6,5)to (9,4);
					\draw[line width=0.5mm] (6,2) to (9,1);
					
					\draw [line width=0.5mm] (-5.5,2) to (-2.5,2); 
					\draw [line width=0.5mm](0,2) to (0.5,2); 
					\draw [line width=0.5mm] (2.5,2) to (3,2); 
					\draw [line width=0.5mm] (-8,1) to (-5.5,1);
					\draw[line width=0.5mm] (-2.5,1) to (0,1);
					\draw [line width=0.5mm](3,1) to (6,1);
					\draw [line width=0.5mm](-8,0) to (-8,-1);
					\draw [line width=0.5mm](-5.5,-1) to (-5.5,0);
					\draw [line width=0.5mm](-2.5,-1) to (-2.5,0);
					\draw [line width=0.5mm](0,-1) to (0,0);
					\draw [line width=0.5mm](3,-1) to (3,0);
					\draw [line width=0.5mm](6,-1)to (6,0);
					\draw [line width=0.5mm](9,-1) to (9,0);
					\draw [line width=0.5mm](-8,-2) to (-8,-3);
					\draw [line width=0.5mm](-5.5,-2) to (-5.5,-3);
					\draw [line width=0.5mm](-2.5,-3) to (-2.5,-2);
					\draw [line width=0.5mm](0,-3) to (0,-2);
					\draw [line width=0.5mm](3,-3) to (3,-2);
					\draw [line width=0.5mm](6,-3) to (6,-2);
					\draw [line width=0.5mm](9,-3) to (9,-2);
					\draw [line width=0.5mm](-8,-6) to (-8,-5);
					\draw [line width=0.5mm](-5.5,-6) to (-5.5,-5);
					\draw [line width=0.5mm](-2.5,-6) to (-2.5,-5);
					\draw [line width=0.5mm](0,-6) to (0,-5); 
					\draw [line width=0.5mm](3,-6)to (3,-5);
					\draw [line width=0.5mm](6,-6) to (6,-5);
					\draw [line width=0.5mm](9,-6) to (9,-5);

			\end{tikzpicture}}
		\end{center}
		\caption{Perfect matching in $\mathbf{Cl_2(R)}$}\label{Figure 3}
	\end{figure}
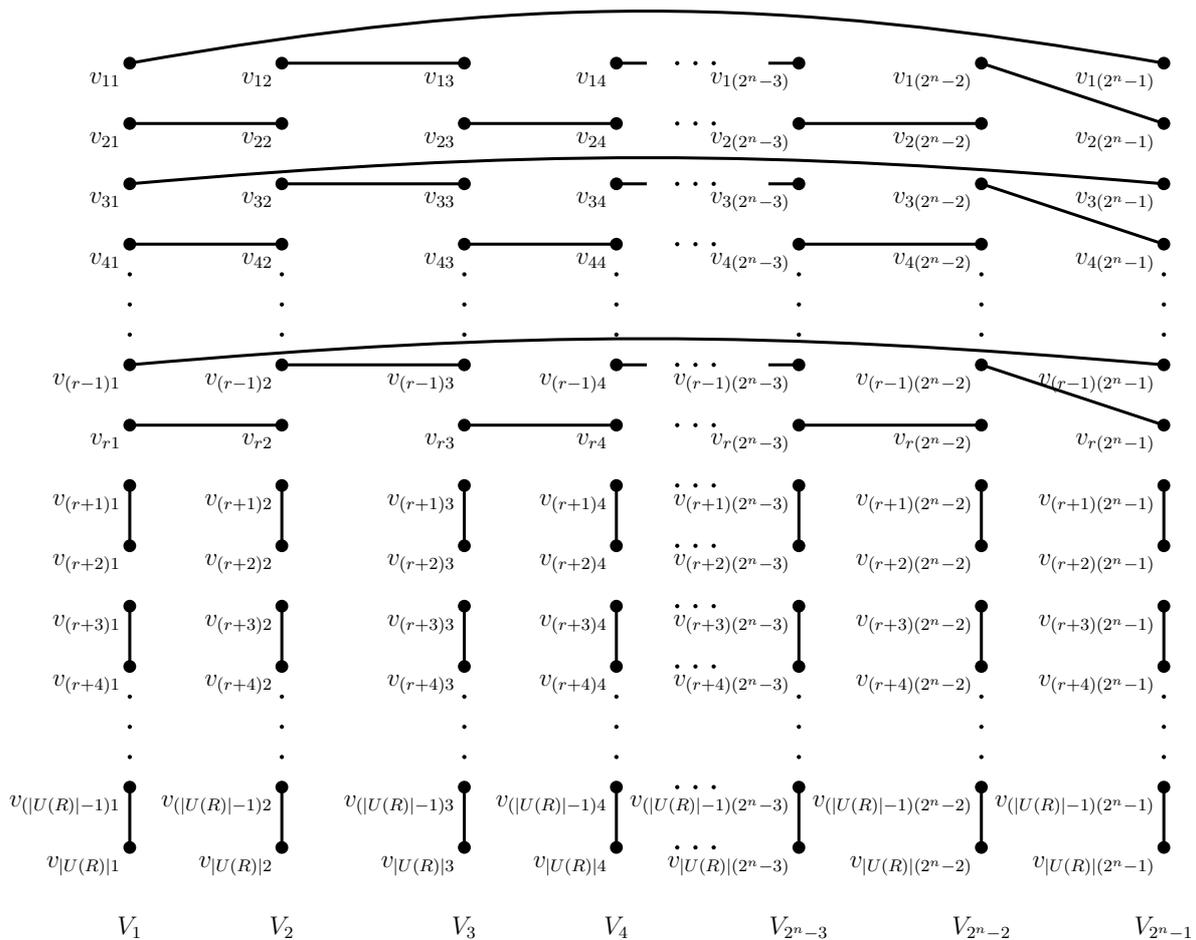
\end{Illustration}
\newpage
\section{Conclusion and further works}
In this paper, we have obtained the Wiener index and matching number of $Cl_2(R)$. Further, we have shown that $Cl_2(R)$ has a perfect matching. Also, we have given a formula for the number of self-invertible units in ring $\mathbb{Z}_n$.

In some sense, the unit product graph and the idempotent graph are related to the clean graph of rings. One can define a product of graphs such that the product of the idempotent graph and the unit product graph is the clean graph.

\end{document}